\documentclass[
	xcolor=x11names,
	dvipsnames,
	a4paper,
	twoside,
]{amsart}

\usepackage[l2tabu, orthodox]{nag}

\usepackage[
	backend=bibtex, 
	style=alphabetic
]{biblatex}
\usepackage{amssymb, mathrsfs}

\usepackage[
	bookmarks=true,
	colorlinks=true,
	linkcolor=blue,
	citecolor=blue,
	urlcolor=red
]{hyperref}
\usepackage[capitalise]{cleveref}

\usepackage{graphicx}
\usepackage[export]{adjustbox} 
\usepackage{ytableau}
\usepackage{dynkin-diagrams}
\usepackage{tikz-cd}

\theoremstyle{plain}
\newtheorem{Theorem}{Theorem}[section]

\newtheorem{Lemma}[Theorem]{Lemma}
\newtheorem{Corollary}[Theorem]{Corollary}

\theoremstyle{definition}
\newtheorem{Definition}[Theorem]{Definition}
\newtheorem{Example}[Theorem]{Example}
\newtheorem{Remark}[Theorem]{Remark}

\usepackage[shortlabels]{enumitem}

\newcommand{\bbC}{\mathbb{C}}

\newcommand{\bbZ}{\mathbb{Z}}
\newcommand{\bbN}{\mathbb{N}}
\newcommand{\bbQ}{\mathbb{Q}}
\newcommand{\bbG}{\mathbb{G}}

\newcommand{\bQ}{\mathbf{Q}}
\newcommand{\bR}{\mathbf{R}}
\newcommand{\bS}{\mathbf{S}}
\newcommand{\bT}{\mathbf{T}}

\newcommand{\scS}{\mathscr{S}}

\newcommand{\cA}{\mathcal{A}}
\newcommand{\cB}{\mathcal{B}}

\newcommand{\cO}{\mathcal{O}}

\newcommand{\cW}{\mathcal{W}}

\newcommand{\cV}{\mathcal{V}}
\newcommand{\cM}{\mathcal{M}}

\newcommand{\fg}{\mathfrak{g}}

\newcommand{\fD}{\mathfrak{D}}
\newcommand{\fS}{\mathfrak{S}}

\newcommand{\xto}[1]{\xrightarrow{#1}}
\newcommand{\injto}{\hookrightarrow}

\newcommand{\Hom}{\operatorname{Hom}}
\newcommand{\Span}{\operatorname{Span}}

\renewcommand{\Im}{\operatorname{Im}}

\newcommand{\GL}{\operatorname{GL}}
\newcommand{\SL}{\operatorname{SL}}
\newcommand{\SO}{\operatorname{SO}}
\newcommand{\fgl}{\mathfrak{gl}}
\newcommand{\fsl}{\mathfrak{sl}}

\newcommand{\wt}{\operatorname{wt}}
\newcommand{\ch}{\operatorname{ch}}
\newcommand{\Supp}{\operatorname{Supp}}

\newcommand{\flag}{\text{flag}}
\newcommand{\parity}{\operatorname{parity}}
\newcommand{\hw}{\mathrm{h.w.}}

\newcommand{\Gr}{\operatorname{\mathsf{Gr}}}
\newcommand{\innprod}[1]{\langle #1 \rangle}
\newcommand{\abs}[1]{\left\lvert #1 \right\rvert}
\newcommand{\integrable}{\mathrm{int}}
\newcommand{\ptimes}{\dot{\times}}
\newcommand{\up}{\operatorname{\mathsf{up}}}
\newcommand{\down}{\operatorname{\mathsf{down}}}
\newcommand{\chardet}{\operatorname{\det}}
\newcommand{\Part}{\operatorname{Part}}

\definecolor{darkred}{rgb}{0.7,0,0} 
\newcommand{\defn}[1]{{\color{darkred}\emph{#1}}} 

\begin{document}

\title{A Demazure Character Formula for the Product Monomial Crystal}
\author{Joel Gibson}
\email{joel.gibson@sydney.edu.au}
\date{\today}

\begin{abstract}
	The product monomial crystal was defined by Kamnitzer, Tingley, Webster, Weekes, and Yacobi for any semisimple simply-laced Lie algebra $\fg$, and depends on a collection of parameters $\bR$. We show that a family of truncations of this crystal are Demazure crystals, and give a Demazure-type formula for the character of each truncation, and the crystal itself. This character formula shows that the product monomial crystal is the crystal of a generalised Demazure module, as defined by Lakshmibai, Littelmann and Magyar. In type $A$, we show the product monomial crystal is the crystal of a generalised Schur module associated to a column-convex diagram depending on $\bR$.
\end{abstract}

\maketitle

\tableofcontents

\section{Introduction}

Let $G$ be a semisimple simply-laced algebraic group over $\bbC$, for example $\SL_n$ (type $A_{n-1}$), $\SO(2n)$ (type $D_n$), or one of the exceptional types $E_6, E_7, E_8$.
The aim of this paper is to study a family of finite dimensional representations of $G$ which appear in the study of slices to Schubert varieties in the affine Grassmannian.
Our main theorem provides a Demazure-type character formula for these representations.
In type $A$ we show that these representations are related to generalised Schur modules, and give an explicit realisation for the crystal of a generalised Schur module.
We now discuss our motivations and results in more detail.

Our first motivation comes from the representation theory of algebras which quantise slices to Schubert varieties in the affine Grassmannian.
Let $G$ be a reductive algebraic group with Langlands dual $G^\vee$.
The affine Grassmannian $\Gr$ of $G^\vee$ is a Poisson ind-variety which plays an important role in geometric representation theory (e.g. \cite{mirkovicGeometricLanglandsDuality2007}).
$\Gr$ is stratified by spherical orbits $\Gr=\bigsqcup_{\lambda\in P^+} \Gr^\lambda$, where $\lambda$ ranges over the dominant weights $P^+$ of $G$. The closure relation on these strata is given by the positive root ordering on $P^+$: $\overline{\Gr^\lambda}=\bigcup_{\mu\leq\lambda}\Gr^\mu$. Fixing a pair $\mu \leq  \lambda$ in $X^+$, consider the \emph{transversal slice} $\overline{\cW}_\mu^{\lambda}$ to $\Gr^\mu$ in $\overline{\Gr^\lambda}$. These slices inherit a Poisson structure, and under the geometric Satake correspondence they (non-canonically) geometrise weight spaces of the irreducible representation of $\fg$ \cite{bravermanPursuingDoubleAffine2010}.  

In \cite{kamnitzerYangiansQuantizationsSlices2014} the authors initiate a program to construct quantisations of $\overline{\cW}_\mu^{\lambda}$, and study the representation theory of the resulting algebras. These algebras are called \emph{truncated shifted Yangians} and denoted $Y_\mu^\lambda(\bR)$.  Here 
$\bR \in \Pi_{i \in I}$
$\mathbb{C}^{\lambda_i}/S_{\lambda_i}$
is a deformation parameter,
where $\lambda_i=\langle \lambda, \alpha_i^\vee \rangle$.  

Modules over truncated shifted Yangians naturally afford a highest weight theory, leading to a category $\cO(Y_\mu^\lambda(\bR))$, which is the ``algebraic category $\cO$'' in the sense of \cite{bradenQuantizationsConicalSymplectic2014}. This category plays an important role in the ``symplectic duality'' program of Braden-Licata-Proudfoot-Webster.  In \cite{kamnitzerCategoryAffineGrassmannian2019} $\cO(Y_\mu^\lambda(\bR))$ was recently used to prove the categorical symplectic duality between Nakajima quiver varieties and the slices in the affine Grassmannian.

The representation theory of    $Y_\mu^\lambda(\bR)$ easily reduces to the case where $\bR$ consists of integers satisfying certain parity conditions (see \cref{section:definition-of-both-crystals} below). It is conjectured in \cite{kamnitzerHighestWeightsTruncated2019}, and proven in \cite{kamnitzerCategoryAffineGrassmannian2019}, that the sum $\cV(\bR)=\bigoplus_\mu  \cO(Y_\mu^\lambda(\bR))$ carries a categorical $\fg$-action in the sense of Khovanov-Lauda and Rouquier \cite{khovanovDiagrammaticApproachCategorification2009,rouquier2KacMoodyAlgebras2008,khovanovDiagrammaticApproachCategorification2011}.  Therefore the (complexified) Grothendieck group of this category is a representation of $\fg$.  This representation is our main object of study: $$V(\bR)=K^\mathbb{C}(\cV(\bR)).$$  

While it is known that for generic (respectively singular) parameters $V(\bR)$ is isomorphic to a tensor product of fundamental representations (respectively a single irreducible representation), the representation $V(\bR)$ in general is quite mysterious. Let $B(\bR)$ be the crystal of $V(\bR)$, which is called the \emph{product monomial crystal} due to its realisation as a subcrystal of Nakajimas's monomial crytal \cite{kamnitzerHighestWeightsTruncated2019}. Our first main result provides 
an explicit character formula for the crystal $B(\bR)$, in terms of multiplications by dominant weights and application of isobaric Demazure operators (\cref{theorem:inductive-character-formula}).

Our proof relies on defining \textit{truncations} of the crystal $B(\bR)$, which are certain subsets of $B(\bR)$ (\cref{section:definition-truncation}). Each of these truncations is described globally, however we show that ``nearby'' truncations are related via a crystal-instrisic operation, the \textit{extension of strings}. Determining a path of nearby truncations from the smallest truncation to the largest gives a Demazure character formula (\cref{sec:character-formula}). We also show that each truncation is a disjoint sum of Demazure crystals (\cref{section:truncations-are-demazure}), which is not obvious from the global description of the truncation.

As a consequence of the categorification described above, the elements of weight $\mu$ in $B(\bR)$ are in bijection with the simple highest weight modules of $Y_\mu^\lambda(\bR)$.  In fact the crystal structure provides even more refined data, for example the highest-weight elements of weight $\mu$ in the crystal correspond precisely to the \emph{finite dimensional} simple $Y_\mu^\lambda(\bR)$-modules \cite[Proposition 3.17]{kamnitzerHighestWeightsTruncated2019}.
Thus, our character formula (\cref{theorem:inductive-character-formula}) provides combinatorial information about the representation theory of $Y_\mu^\lambda(\bR)$.

Our second motivation comes from the study of generalised Schur modules. Recall that for each partition $\lambda$, there is an endofunctor on vector spaces called the \textit{Schur functor} $\scS_\lambda$, and when $V$ is the basic representation of $\GL(V)$ over $
\mathbb{C}$, the representation $\scS_\lambda(V)$ is irreducible, of highest weight $\lambda$. The modules $\scS_\lambda(V)$ are well-studied and are called \textit{Schur modules}.

The partition $\lambda$ can be thought of as a Young diagram, a configuration of boxes in the plane, or a finite subset of $\bbN \times \bbN$ of a special form. In fact, to \textit{any} finite subset $D \subseteq \bbN \times \bbN$, called a \textit{diagram}, we can associate a generalised Schur functor $\scS_D$, and hence a generalised Schur module $\scS_D(V)$. Considerably less is known about these functors in complete generality, though there have been some combinatorics developed for the class of ``percentage-avoiding'' diagrams \cite{magyarBorelWeilTheorem1998,reinerPercentageAvoidingNorthwestShapes1998}.

Our second main result describes the crystal of $\scS_D$ in the case where $D$ is \textit{column-convex} (see \cref{remark:column-convex}). More precisely, we associate a parameter $\bR$ to $D$ and prove that $B(\bR)$ is the crystal of $\scS_D$ (\cref{theorem:crystal-of-schur-module}). This provides a model for the crystal of $\scS_D$ in terms of Nakajima monomials (the other known model for this crystal is due to Laksmibai, Littelmann, and Magyar \cite{lakshmibaiStandardMonomialTheory2002} using Demazure operators).

To prove \cref{theorem:crystal-of-schur-module}, we compare our character formula (\cref{theorem:inductive-character-formula}) for $B(\bR)$ in type $A_n$ to a character formula for $\scS_D(\bbC^{n + 1})$ given by Reiner and Shimozono \cite[Theorem 23]{reinerPercentageAvoidingNorthwestShapes1998}, which shows that they have the same character when $n$ is taken large enough compared to $D$. We then derive some stability results for $B(\bR)$ (\cref{section:stability}), which imply that $B(\bR)$ is the character of $\scS_D(\bbC^{n+1})$ for any $n$.

\cref{theorem:crystal-of-schur-module} implies that in type $A$, the category $\cV(\bR)$ defined by the truncated shifted Yangians are categorifications of generalised Schur modules associated to column-convex diagrams. We note that our result shows that any skew Schur module is categorified by $\cV(\bR)$, for some $\bR$.

\subsection{Acknowledgements} I would like to thank my advisor Oded Yacobi for suggesting this project, and I am grateful for his support, guidance and patience throughout. I'd like to thank Peter Tingley and Alex Weekes for many helpful conversations while they were visiting Sydney. I'd also like to thank Nicolle Gon\'{z}alez and Travis Scrimshaw for many illuminating conversations concerning Demazure crystals and Victor Reiner for answering all of my questions about generalised Schur modules. This research was supported by an Australian Postgraduate Award; the results of this paper will appear in my PhD thesis.
\section{Background} \label{section:background}

\subsection{Notation} \label{section:notation}

Let $G$ be a reductive algebraic group over $\bbC$, equipped with a pinning $T \subseteq B \subseteq G$ where $B$ is a Borel subgroup and $T$ is a maximal split torus. Let $I$ be a set such that $|I|$ is the semisimple rank of $G$. This determines the following combinatorial data:
\begin{enumerate}
	\item The \defn{weight lattice} $P = \Hom(T, \bbG_m)$.
	\item The \defn{simple roots} $\alpha_i \in P$ for all $i \in I$.
	\item The \defn{coweight lattice} $P^\vee = \Hom(\bbG_m, T)$.
	\item The \defn{simple coroots} $\alpha_i^\vee \in P^\vee$ for all $i \in I$.
	\item The \defn{perfect pairing} $\innprod{-, -}: P^\vee \times P \to \bbZ$.
	\item The \defn{dominant weights} $P_+ = \{\lambda \in P \mid \innprod{\alpha_i^\vee, \lambda} \geq 0 \text{ for all } i \in I\}$.
	\item The \defn{simple reflection} $s_i \in \GL(P)$ defined by $s_i(\lambda) = \lambda - \innprod{\alpha_i^\vee, \lambda} \alpha_i$.
	\item The \defn{Weyl group} $W = \innprod{s_i \mid i \in I} \subseteq \GL(P)$, and its \defn{longest element} $w_\circ$.
\end{enumerate}
Because of the proof of \cref{theorem:monomial-crystal-is-crystal} relies on the theory of Nakajima quiver varieties, we assume throughout that $G$ is \defn{simply-laced}, meaning that the Cartan matrix $\innprod{\alpha_i^\vee, \alpha_j}_{i, j}$ is symmetric with all off-diagonal entries either 0 or $-1$. Define the \defn{Dynkin diagram} of $G$ as the simple graph with vertex set $I$, where $i \sim j$ if and only if $\innprod{\alpha_i^\vee, \alpha_j} = -1$. We may then fix a two-colouring $I = I_0 \sqcup I_1$ of the Dynkin diagram, and we say that the \defn{parity} of a vertex $i \in I$ is \defn{even} if $i \in I_0$ and \defn{odd} if $i \in I_1$.

A finite-dimensional module $V$ over $G$ decomposes into \defn{weight spaces} $V = \bigoplus_{\lambda \in P} V_\lambda$, where $V_\lambda = \{ v \in V \mid t \cdot v = \lambda(t) v \text{ for all } t \in T\}$. Let $\bbZ[P]$ denote the group algebra of $P$ written multiplicatively, so that $e^\lambda e^\mu = e^{\lambda + \mu}$. Then the \defn{formal character} of the module $V$ is the sum $\ch V = \sum_{\lambda \in P} (\dim V_\lambda) e^\lambda \in \bbZ[P]^W$, where $\bbZ[P]^W$ denotes the subalgebra invariant under the action $s_i \cdot e^\lambda = e^{s_i \lambda}$ of $W$.

For each dominant weight $\lambda \in P^+$, let \defn{$V(\lambda)$} denote the irreducible module of highest-weight $\lambda$. The category of finite-dimensional $G$-modules is semisimple, with the modules $\{V(\lambda) \mid \lambda \in P\}$ forming a complete irredudant list of simple objects. The characters of the $V(\lambda)$ for $\lambda \in P^+$ form a basis for $\bbZ[P]^W$, and hence the isomorphism class of any finite-dimensional $G$-module is determined entirely by its character.

\begin{Remark}
	We have restricted to the case of reductive $G$ in order to simplify the exposition, however our character formula \cref{theorem:inductive-character-formula} holds in the more general Kac-Moody setting provided that the Dynkin diagram is still simply-laced and bipartite. In order to state the analagous results in this more general setting, one would replace the category of finite-dimensional $G$-modules by the category $\cO^\integrable$ for the corresponding quantum group. In particular, our results apply to finite types $A, D, E$ and their untwisted affinisations, excluding $A_1^{(1)}$ which is not simply-laced, and $A_{n}^{(1)}$ for $n$ odd, which is not bipartite.
\end{Remark}

\subsection{Crystals} \label{section:crystal-preliminaries}

Crystals were unearthed by Kashiwara \cite{kashiwaraCrystalizingTheqanalogueUniversal1990, kashiwaraCrystalBasesAnalogue1991a, kashiwaraCrystalBasesModified1994}. There is a rather general notion of a $G$-crystal, however we will only require the notion of an \textit{upper-seminormal} crystal, for which we can give some simplified axioms and definitions. We follow the exposition of \cite[Section~2]{josephDecompositionTheoremDemazure2003} for this section.

An upper-seminormal \defn{abstract $G$-crystal} is a set $B$, together with a \defn{weight function} $\wt \colon B \to P$, and for each $i \in I$, \defn{crystal operators} $e_i, f_i \colon B \to B \sqcup \{0\}$ and maps $\varepsilon_i, \varphi_i \colon B \to \mathbb{Z}$ satisfying the following axioms:
\begin{enumerate}
	\item $\varphi_i(b) = \varepsilon_i(b) + \innprod{\alpha_i^\vee, \wt b}$ for all $i \in I, b \in B$.
	\item $e_i b = b'$ if and only if $b = f_i b'$, for all $b, b' \in B$.
	\item For all $b \in B, i \in I$ such that $e_i(b) \in B$, we have $\wt(e_i(b)) = \wt b + \alpha_i$.
	\item $\varepsilon_i(b) = \max\{k \geq 0 \mid e_i^k b \in B\}$.
\end{enumerate}
By the above axioms, the data of an upper-seminormal abstract crystal is entirely determined by $(B, \wt, (e_i)_{i \in I})$. The upper-seminormal abstract crystal $B$ is called a \defn{seminormal abstract crystal} if it additionally satisfies $\varphi_i(b) = \max\{k \geq 0 \mid f_i^k b \in B\}$ for all $i \in I, b \in B$.

Each crystal defines a \defn{crystal graph}, a directed graph on the vertex set $B$, with an $i$-labelled edge from $b$ to $b'$ whenever $f_i(b) = b'$. The edge-labelled graph is equivalent to the data of the $e_i$ or $f_i$, and hence an upper-seminormal abstract crystal is determined entirely by its weight function and graph. We say that $\cB$ is \defn{connected} if the underlying undirected graph of its crystal graph is connected. An element $b \in B$ is called \defn{primitive} if $e_i(b) = 0$ for all $i \in I$, or equivalently if it has no incoming edges in the crystal graph. An element $b \in B$ is called \defn{highest-weight} if it is both primitive, and there is a directed path in the crystal graph from $b$ to every element of $B$. 

There are two different rules for forming the tensor product of two abstract crystals, we use the convention from \cite{kashiwaraCrystalBasesModified1994}. If $B_1, B_2$ are abstract $G$-crystals, then their tensor product $B_1 \otimes B_2$ has underlying set the Cartesian product $B_1 \times B_2$, with pairs of elements written $b_1 \otimes b_2$, the convention that $0 \otimes b_2 = b_1 \otimes 0 = 0$, and maps given by
\begin{align}
	\wt(b_1 \otimes b_2) &= \wt(b_1) + \wt(b_2) \\
	\varepsilon_i(b_1 \otimes b_2) &= \max\{\varepsilon_i(b_1), \varepsilon_i(b_2) - \innprod{\alpha_i^\vee, \wt b_1}\} \\
	\varphi_i(b_1 \otimes b_2) &= \max\{\varphi_i(b_2), \varphi_i(b_1) + \innprod{\alpha_i^\vee, \wt b_2}\} \\
	f_i(b_1 \otimes b_2) &= \begin{cases}
		f_i b_1 \otimes b_2 & \text{if } \varphi_i(b_1) > \varepsilon_i(b_2) \\
		b_1 \otimes f_i b_2 & \text{if } \varphi_i(b_1) \leq \varepsilon_i(b_2)
	\end{cases} \\
	e_i(b_1 \otimes b_2) &= \begin{cases}
		e_i b_1 \otimes b_2 & \text{if } \varphi_i(b_1) \geq \varepsilon_i(b_2) \\
		b_1 \otimes e_i b_2 & \text{if } \varphi_i(b_1) < \varepsilon_i(b_2)
	\end{cases} \label{equation:ei-tensor-first}
\end{align}
This gives the category of abstract crystals the structure of a monoidal category (the tensor product of two upper-seminormal crystals is again upper-seminormal).

It has been shown \cite{kashiwaraCrystalBasesAnalogue1991a} that for each $\lambda \in P^+$, the highest-weight module $V(\lambda)$ admits a seminormal crystal base $\cB(\lambda)$. We say that an abstract $G$-crystal is simply a \defn{$G$-crystal} if it is a disjoint union of various $\cB(\lambda)$. Each $\cB(\lambda)$ is connected as a graph, with a unique vertex $b_\lambda \in \cB(\lambda)$ satisfying $e_i(b_\lambda) = 0$ for all $i \in I$. Such a vertex $b_\lambda$ is called \defn{highest weight}, since it is both \defn{primitive} (meaning it is killed by all the $e_i$) and also generates the whole of $\cB(\lambda)$ under the $f_i$ operators.

The subcategory of $\fg$-crystals is closed under tensor product, and the decomposition multiplicities of the $\cB(\lambda)$ agree with those of the $V(\lambda)$:
\[
	[\cB(\lambda) : \cB(\mu) \otimes \cB(\nu)] = [V(\lambda) : V(\mu) \otimes V(\nu)]
	\quad
	\text{ for all } \lambda, \mu, \nu \in P^+.
\]

For each $\lambda \in P$ we denote by $B_\lambda = \{b \in B \mid \wt b = \lambda \}$ the $\lambda$-weight elements of $B$.  The \defn{formal character} of an abstract $G$-crystal is the sum $\ch B = \sum_{\lambda \in P} \abs{B_\lambda} e^\lambda \in \bbZ[P]$. If the abstract $G$-crystal is seminormal, then $\ch B \in \bbZ[P]^W$. For each $\lambda \in P^+$, we have $\ch V(\lambda) = \ch \cB(\lambda)$.

\subsection{Demazure modules and Demazure crystals}
\label{section:demazure-crystals}

Fix a $\lambda \in P^+$. The elements of the Weyl group orbit $W \cdot \lambda$ are called the \defn{extremal weights} of $V(\lambda)$, and their corresponding weight spaces $V(\lambda)_{w \lambda}$ are all one-dimensional. The $B$-submodule generated by $V(\lambda)_{w \lambda}$ is called the \defn{Demazure module} $V_w(\lambda)$. As vector spaces we have $V_e(\lambda) = V(\lambda)_\lambda$, the one-dimensional highest-weight space, and $V_{w_\circ}(\lambda) = V(\lambda)$ the whole module. We say that $V_w(\lambda)$ is the Demazure module of \defn{Demazure lowest weight} $w \lambda$.

For each $\mu \in P$, let $D(\mu)$ denote the Demazure module of Demazure lowest weight $\mu$. The collection $\{D(\mu) \mid \mu \in P\}$ is a complete irredundant list of Demazure modules, and furthermore the \defn{Demazure characters} $\{\ch D(\mu) \mid \mu \in P\}$ form a basis for $\bbZ[P]$. (The triangularity property $\ch D(\mu) = e^\mu + \sum_{\nu > \mu} c_\nu e^\nu$ shows that the Demazure characters are linearly independent, while the fact that they span follows from the fact that $\bbZ[P]$ is a limit of finite-dimensional subspaces of the form $\Span_\bbZ{X}$ for $X \subseteq P$ Weyl-invariant).

A character formula for the Demazure modules was first given in \cite{demazureDesingularisationVarietesSchubert1974}, with a more recent proof in the arbitrary symmetrisable Kac-Moody setting appearing in \cite{kashiwaraCrystalBaseLittelmann1993}. For each $i \in I$, define the $\bbZ$-linear \defn{Demazure operator} $\pi_i: \bbZ[P] \to \bbZ[P]$ by
\begin{align}
	\label{equation:demazure-operator}
	\pi_i(e^\lambda)
	&= \frac{e^{\lambda} - e^{s_i \lambda - \alpha_i}}{1 - e^{- \alpha_i}} \\
	&= \begin{cases}
		e^\lambda + e^{\lambda - \alpha_i} + e^{\lambda - 2 \alpha_i} + \cdots + e^{s_i \lambda}
			& \text{if } \innprod{\alpha_i^\vee, \lambda} \geq 0, \\
		0
			& \text{if } \innprod{\alpha_i^\vee, \lambda} = -1, \\
		- e^{\lambda + \alpha_i} - e^{\lambda + 2\alpha_i} - \cdots - e^{s_i \lambda - \alpha_i}
			& \text{if } \innprod{\alpha_i^\vee, \lambda} \leq -2.
	\end{cases}
\end{align}
The Demazure operators define a 0-Hecke action on $\bbZ[P]$, that is to say each is idempotent ($\pi_i^2 = \pi_i$ for all $i \in I$) and they satisfy the braid relations (if $i \sim j$ then $\pi_i \pi_j \pi_i = \pi_j \pi_i \pi_j$, and if $i \not \sim j$ then $\pi_i$ and $\pi_j$ commute). The \textit{Demazure character formula} states that when $(i_1, \ldots, i_r)$ is a reduced decomposition for $w$, then
\begin{equation}
	\ch V_w(\lambda) = \pi_{i_1} \cdots \pi_{i_r} e^\lambda.
\end{equation}

It was shown \cite{kashiwaraCrystalBaseLittelmann1993} that each Demazure module $V_w(\lambda)$ for $\lambda \in P^+$ admits a crystal base $\cB_w(\lambda)$, called a \defn{Demazure crystal} (although $\cB_w(\lambda)$ is in general an \textit{abstract} $G$-crystal rather than a crystal). Moreover, the crystal $\cB_w(\lambda)$ can be obtained from $\cB(\lambda)$ in the following way. Define the \defn{extension of $i$-strings} operator $\fD_i$, which acts on a subset $X \subseteq B$ of an abstract $G$-crystal:
\begin{equation}
	\fD_i(X) = \bigcup_{n \geq 0} \{f_i^n(x) \mid x \in X\} = \{b \in B \mid e_i^n(b) \in X \text{ for some } n \in \bbN\}.
\end{equation}
The operators $\fD_i$ satisfy $X \subseteq \fD_i X \subseteq B$, and $\fD_i( \fD_i X ) = \fD_i X$. Then let $(i_1, \ldots, i_r)$ be a reduced decomposition for $w \in W$, and we have
\begin{equation} \label{equation:demazure-crystal-definition}
	\cB_w(\lambda) = \fD_{i_1} \cdots \fD_{i_r} \{ b_\lambda \},
\end{equation}
where $b_\lambda \in \cB(\lambda)$ is the highest-weight element. Note that \cref{equation:demazure-crystal-definition} only defines a subset of $\cB(\lambda)$: we equip this subset with its canonical abstract upper-seminormal crystal structure coming from the restrictions of $\wt$ and $e_i$ for $i \in I$.

\subsection{Multisets}
\label{section:multisets}

We will use multisets throughout the paper. Multisets will always be denoted using boldface type, such as $\bR, \bS, \bT$, or $\bQ$. Given a set $X$, a \defn{multiset based in $X$} is a function $\bR \colon X \to \bbN$, where we write $\bR[x]$ for the value of $\bR$ at $x \in X$, henceforth called the \defn{multiplicity} of $x$ in $\bR$. The \defn{support} of $\bR$ is the subset $\Supp(\bR) = \{x \in X \mid \bR[x] > 0\} \subseteq X$, and a multiset is \defn{finite} if its support is finite. Any summations or products over $\bR$ are taken with multiplicity, so for example if $f \colon X \to G$ is a function into an abelian group (written multiplicatively) and $\bR$ is a finite multiset based in $X$, then $\prod_{x \in \bR} f(x) := \prod_{x \in X} f(x)^{\bR[x]}$. If $\bR$ and $\bQ$ are multisets based in $X$, their \defn{multiset union} is the function $\bR + \bQ$. We say $\bQ$ is a \defn{sub-multiset} of $\bR$ if $\bQ[x] \leq \bR[x]$ for all $x \in X$, and in this case, their \defn{multiset difference} is the function $\bR - \bQ$.

The notation we use for multisets is similar to set notation, with exponents denoting multiplicity. For example, if $X = \{x, y, z\}$ is a base set, then $\bR = \{x^2, y\}$ denotes a multiset $\bR$ based in $X$ where $x$ appears with multiplicity $2$, and $y$ appears with multiplicity $1$ (or treating $\bR$ as a function $\bR \colon X \to \mathbb{N}$, we have $\bR[x] = 2, \bR[y] = 1, \bR[z] = 0$).
\section{Definition of the product monomial crystal} \label{section:definition-of-both-crystals}

The product monomial crystal is defined as a certain subcrystal of the \textit{Nakajima monomial crystal}. The Nakajima monomial crystal is not only a crystal, but also has an abelian group operation given by multiplication of monomials. The product monomial crystal will be the monomial-wise product of certain subcrystals of the Nakajima monomial crystal.

\subsection{The Nakajima monomial crystal} \label{section:monomial-crystal}

Let $G$ be a pinned reductive group as in \cref{section:notation}, we now define the Nakajima monomial crystal $\cM(G)$ as in \cite[Section~2]{hernandezLevelMonomialCrystals2006}. Let $\bbZ\{I \times \bbZ\}$ denote the free abelian group of \textit{monomials} in the variables $\{y_{i, c} \mid i \in I, c \in \bbZ\}$. Let $\cA(G) = P \times \bbZ\{I \times \bbZ\}$ be the product of abelian groups, written such that a typical element $p \in \cA$ is of the form
\begin{equation}
	p = e^{\wt(p)} \prod_{(i, c) \in I \times \bbZ} y_{i, c}^{p[i, c]},
\end{equation}
for some element $\wt(p) \in P$ and coefficients $p[i, c] \in \bbZ$, finitely many of which are nonzero. For each $i \in I$ and $c \in \bbZ$, define the auxiliary monomial
\begin{equation}
	z_{i, k} = e^{\alpha_i} y_{i, k + 1} y_{i, k - 1} \prod_{j \sim i} y_{j, k}^{-1}.
\end{equation}

\begin{Definition}
	\label{definition:monomial-crystal}
	The \defn{Nakajima monomial crystal} $\cM(G)$ is defined to be the submodule of $\cA(G)$ satisfying the two conditions
	\begin{enumerate}
		\item $\innprod{\alpha_i^\vee, \wt(p)} = \sum_{c \in \bbZ} p[i, c]$ for all $i \in I$, and
		\item $p[i, c] = 0$ if $i$ and $c$ have opposite parities.
	\end{enumerate}
	For each monomial $p \in \cM(\fg)$, define
	\begin{enumerate}
		\item $\varphi_i^k = \sum_{l \geq k} p[i, l]$, the \defn{upper column sum}.
		\item $\varphi_i(p) = \max_k \varphi_i^k(p)$, the \defn{largest upper column sum}.
		\item $\varepsilon_i^k(p) = -\sum_{l \leq k} p[i, l]$, the \defn{negated lower column sum}.
		\item $\varepsilon_i(p) = \max_k \varepsilon_i^k(p)$, the \defn{largest negated column sum}.
		\item If $\varphi_i(p) \neq 0$, let $F_i(p) = \max\{k \in \bbZ \mid \varphi_i^k(p) = \varphi_i(p)\}$, the largest $k$ maximising $\varphi_i^k(p)$.
		\item If $\varepsilon_i(p) \neq 0$, let $E_i(p) = \min\{k \in \bbZ \mid \varepsilon_i^k(p) = \varepsilon_i(p)\}$, the smallest $k$ maximising $\varepsilon_i^k(p)$.
	\end{enumerate}
	The crystal structure on $\cM(G)$ is then defined by $\wt$, $\varepsilon_i$, and $\varphi_i$ as above, and
	\begin{equation}
		\begin{aligned}
			e_i(p) = \begin{cases}
				0 &\text{if } \varepsilon_i(p) = 0 \\
				p z_{i, E_i(p)} & \text{otherwise}
			\end{cases}
			&&&&&
			f_i(p) = \begin{cases}
				0 &\text{if } \varphi_i(p) = 0 \\
				p z_{i, F_i(p) - 2}^{-1} & \text{otherwise}.
			\end{cases}
		\end{aligned}
	\end{equation}
\end{Definition}
It is routine to verify that $(\cM(G), \wt, \varepsilon_i, \varphi_i, e_i, f_i)$ is an abstract seminormal $G$-crystal.

\begin{Theorem}[{\cite{kashiwaraRealizationsCrystals2002}}] \label{theorem:monomial-crystal-is-crystal}
	The monomial crystal $\cM(G)$ is a $G$-crystal.
\end{Theorem}

Because of the parity condition (2) appearing in \cref{definition:monomial-crystal}, it is convenient to introduce the following notation. Let $I \ptimes \bbZ \subseteq I \times \bbZ$ denote the subset of parity-respecting pairs
\begin{equation}
	I \ptimes \bbZ = \{(i, c) \in I \times \bbZ \mid i \text{ and } c \text{ have the same parity}\}.
\end{equation}
The monomial crystal without condition (2) coincides with the crystal defined in \cite[Section~3]{kashiwaraRealizationsCrystals2002}. Restricting to monomials $y_{i, c}$ for $(i, c) \in I \ptimes \bbZ$ forms a ``good'' subset in the sense of \cite[Proposition~3.1]{kashiwaraRealizationsCrystals2002}, hence \cref{theorem:monomial-crystal-is-crystal}.

We often picture elements of $\cM(G)$ in the following way. Place the Dynkin diagram $I$ in the plane, then place the grid $I \ptimes \bbZ$ above the Dynkin diagram as an infinite strip of points. A monomial $p \in \cM(G)$ is a finitely supported assignment $(i, c) \mapsto p[i, c]$ of points to integers, together with a weight $\wt p \in P$. The statistics $\varphi_i^k(p)$ and $\varepsilon_i^k(p)$ can then be pictured as a sum over points in a half-infinite column $i$. An example for $G = \SL_6$ is shown in \cref{figure:phi-epsilon-computation}.

\begin{figure}[ht]
	\centering
	\includegraphics[width=0.4\linewidth]{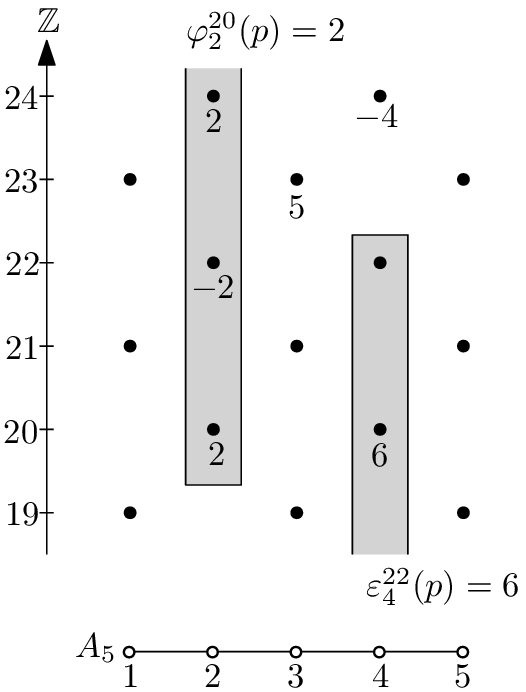}
	\caption{The group $G = \SL_6$ has Dynkin diagram $A_5$, a path on 5 vertices. Pictured above is the monomial
	\[ p = e^{2 \varpi_2 + 5 \varpi_3 + 2 \varpi_4} \cdot y_{2, 22}^{-2} \cdot y_{2, 24}^2 \cdot y_{3,23}^5 \cdot y_{4, 20}^6 \cdot y_{4, 24}^{-4}. \]
	The two shaded regions are showing computations of $\varphi_2^{20}(p)$ and $\varphi_4^{22}(p)$ respectively.}
	\label{figure:phi-epsilon-computation}
\end{figure}

\subsection{The product monomial crystal} \label{section:definition-of-product-monomial-crystal}

From this point onwards, we fix a system $\{\varpi_i\}_{i \in I}$ of \defn{fundamental weights}, meaning any collection of elements satisfying $\innprod{\alpha_i, \varpi_j} = \delta_{ij}$. Note that in general the $\varpi_i$ are members of $P \otimes_\bbZ \bbQ$ rather than the weight lattice $P$, and for general $G$ such a system is not unique, a notable exception being when $G$ is semisimple.

\begin{Definition}
	A \defn{fundamental subcrystal} of $\cM$ is a subcrystal generated by an element of the form $e^{n \varpi_i} y_{i, c}^n$ for some $n \geq 0$ such that $n \varpi_i \in P^+$. Denote this subcrystal by $\cM(i, c)^n$.
\end{Definition}

It is straightforward to check that such a monomial $p = e^{n \varpi_i} y_{i, c}^n$ is highest-weight, and therefore generates a subcrystal of $\cM$ isomorphic to $\cB(n \varpi_i)$ by \cref{theorem:monomial-crystal-is-crystal}. 

As the monomial crystal $\cM$ is a subgroup of $\cA$, it inherits the group operation given by \defn{multiplication of monomials}. Explicity, for $p, q \in \cM$ we have $\wt(p \cdot q) = \wt(p) + \wt(q)$, and $(p \cdot q)[i, c] = p[i, c] + q[i, c]$. For subsets $X, Y \subseteq \cM$, we define their product $X \cdot Y = \{x \cdot y \mid x \in X, y \in Y\}$ as usual, and call this the \defn{monomial-wise product of subsets}.

\begin{Definition}
\label{definition:product-monomial-crystal}
For any finite multiset $\bR$ based in $I \ptimes \bbZ$, define the \defn{product monomial crystal}
\begin{equation}
	\cM(\bR) = \prod_{(i, c)} \cM(i, c)^{\bR[i, c]} \subseteq \cM(G)
\end{equation}
as the monomial-wise product of subsets of the fundamental subcrystals $\cM(i, c)^{\bR[i, c]}$. We set $\cM(\varnothing) = \{1\}$, the trivial monomial.
\end{Definition}

\begin{Remark} \label{remark:omit-weight-term}
	After fixing the system of fundamental weights, the weight of a monomial $p \in \cM(\bR)$ is given by $\wt p = \sum_{i, c} p[i, c] \varpi_i$. So we can safely omit the $e^\lambda$ term from monomials from now on, instead relying on our fixed system of fundamental weights to reconstruct $\lambda$ from the monomial $p$.
\end{Remark}

\begin{Example} \label{example:some-product-monomial-crystals}
	Let $G = \SL_3$, where $I = \{1, 2\}$ with Dynkin diagram $\dynkin[label,label macro/.code={#1}]{A}{2}$.
	\begin{enumerate}
		\item Taking $\bR = \varnothing$ gives $\cM(\bR) = \{1\}$, the trivial crystal.
		\item Taking $\bR = \{(1, 1)\}$ gives $\cM(\bR) = \cM(1, 1) = \{y_{1, 1}, y_{1, 1} z_{1, -1}^{-1}, y_{1, 1} z_{1, -1}^{-1} z_{2, -2}^{-1}\}$, a connected crystal of highest-weight $\varpi_1$.
		\item Taking $\bR = \{(1, 1)^2\}$ gives $\cM(\bR) = \cM(1, 1) \cdot \cM(1, 1)$, which in terms of elements is
		\[ \cM(\bR) = \left\{
			y_{1, 1}^2,\,\,
			y_{1, 1}^2 z_{1, -1}^{-1},\,\,
			y_{1, 1}^2 z_{1, -1}^{-1} z_{2, -2}^{-1},\,\,
			y_{1, 1}^2 z_{1, -1}^{-2},\,\,
			y_{1, 1}^2 z_{1, -1}^{-2} z_{2, -2}^{-1},\,\,
			y_{1, 1}^2 z_{1, -1}^{-2} z_{2, -2}^{-2} 
			\right\}.\] This subset is closed under the crystal operators, and its only highest-weight element is $y_{1, 1}^2$, showing that $\cM(\bR) \cong \cB(\wt y_{1, 1}^2) = \cB(2 \varpi_1)$ as $\SL_3$-crystals.
	\end{enumerate}
	We may denote the monomials above pictorially, as in \cref{figure:phi-epsilon-computation}:
	\[
		y_{1, 1}^a y_{2, 0}^b y_{1, -1}^c y_{2, -2}^d =
		\renewcommand*{\arraystretch}{0.8}
		\renewcommand*{\arraycolsep}{2pt}
		\begin{matrix}
			a & \\
			& b \\
			c & \\
			& d \\
		\end{matrix}
		, \quad\quad
		z_{1, -1} =
		\renewcommand*{\arraystretch}{0.8}
		\renewcommand*{\arraycolsep}{2pt}
		\begin{matrix}
			-1 & \\
			& 1 \\
			-1 & \\
			& 0 \\
		\end{matrix}
		, \quad\quad
		z_{2, -2} =
		\renewcommand*{\arraystretch}{0.8}
		\renewcommand*{\arraycolsep}{2pt}
		\begin{matrix}
			0 & \\
			& -1 \\
			1 & \\
			& -1 \\
		\end{matrix}
	\]
	We then obtain the following pictures for the three connected crystals $\cM(\varnothing)$, $\cM(1, 1)$ and $\cM(1, 1)^2$:

	\[
	\begin{tikzcd}
	\renewcommand*{\arraystretch}{0.8}
	\renewcommand*{\arraycolsep}{2pt}
	\begin{matrix}
		0 & \\
		& 0 \\
		0 & \\
		& 0 \\
	\end{matrix}
	&
	&
	\renewcommand*{\arraystretch}{0.8}
	\renewcommand*{\arraycolsep}{2pt}
	\begin{matrix}
		1 & \\
		& 0 \\
		0 & \\
		& 0 \\
	\end{matrix}
	\ar[r, "f_1"]
	&
	\renewcommand*{\arraystretch}{0.8}
	\renewcommand*{\arraycolsep}{2pt}
	\begin{matrix}
		0 & \\
		& 1 \\
		-1 & \\
		& 0 \\
	\end{matrix}
	\ar[d, "f_2"]
	&
	&
	\renewcommand*{\arraystretch}{0.8}
	\renewcommand*{\arraycolsep}{2pt}
	\begin{matrix}
		2 & \\
		& 0 \\
		0 & \\
		& 0 \\
	\end{matrix}
	\ar[r, "f_1"]
	&
	\renewcommand*{\arraystretch}{0.8}
	\renewcommand*{\arraycolsep}{2pt}
	\begin{matrix}
		1 & \\
		& 1 \\
		-1 & \\
		& 0 \\
	\end{matrix}
	\ar[r, "f_1"]
	\ar[d, "f_2"]
	&
	\renewcommand*{\arraystretch}{0.8}
	\renewcommand*{\arraycolsep}{2pt}
	\begin{matrix}
		0 & \\
		& 2 \\
		-2 & \\
		& 0 \\
	\end{matrix}
	\ar[d, "f_2"]
	\\
	&
	&
	&
	\renewcommand*{\arraystretch}{0.8}
	\renewcommand*{\arraycolsep}{2pt}
	\begin{matrix}
		0 & \\
		& 0 \\
		0 & \\
		& -1 \\
	\end{matrix}
	&
	&
	&
	\renewcommand*{\arraystretch}{0.8}
	\renewcommand*{\arraycolsep}{2pt}
	\begin{matrix}
		1 & \\
		& 0 \\
		0 & \\
		& -1 \\
	\end{matrix}
	\ar[r, "f_1"]
	&
	\renewcommand*{\arraystretch}{0.8}
	\renewcommand*{\arraycolsep}{2pt}
	\begin{matrix}
		0 & \\
		& 1 \\
		-1 & \\
		& 1 \\
	\end{matrix}
	\ar[d, "f_2"]
	\\
	&
	&
	&
	&
	&
	&
	&
	\renewcommand*{\arraystretch}{0.8}
	\renewcommand*{\arraycolsep}{2pt}
	\begin{matrix}
		0 & \\
		& 0 \\
		0 & \\
		& -2 \\
	\end{matrix}
	\end{tikzcd}
	\]
\end{Example}

It is perhaps surprising that the monomial-wise product $\cM(1, 1) \cdot \cM(1, 1)$ turns out to be again a $G$-subcrystal of $\cM(G)$, since the monomial-wise product is not obviously related to the crystal operators. In fact, the subset $\cM(\bR)$ is always a subcrystal of $\cM$, justifying the name product monomial \textit{crystal}. The proof of \cref{theorem:prod-monomial-crystal-is-a-crystal} uses an explicit isomorphism between $\cM(\bR)$ and the crystal defined by a graded Nakajima quiver variety depending on $\bR$. The author does not know of a purely combinatorial proof, and the use of quiver varieties is the reason for the simply-laced restriction on $G$.

\begin{Theorem} \label{theorem:prod-monomial-crystal-is-a-crystal}
	\cite[Corollary 7.8]{kamnitzerHighestWeightsTruncated2019}. For any finite multiset $\bR$ based in $I \ptimes \bbZ$, the set $\cM(\bR)$ is a strict subcrystal of $\cM$. Hence the product monomial crystal $\cM(\bR)$ is a crystal.
\end{Theorem}

\begin{Remark}
	Our notation differs from \cite{kamnitzerHighestWeightsTruncated2019} in three ways. Firstly, they use the symbol $\cB(\bR)$ rather than $\cM(\bR)$ to denote the product monomial crystal. Secondly, they define the monomial crystal only for semisimple Lie algebras, and hence the $e^\lambda$ weight term is missing as explained in \cref{remark:omit-weight-term}. Thirdly, they use a collection $(R_i)_{i \in I}$ of multisets where $R_i$ is a multiset based in $2\bbZ + \parity(i)$; to go between these notations, set $\bR[i, k] = R_i[k]$.
\end{Remark}
 
Define the dominant weight $\wt(\bR) = \sum_{(i, c)} \bR[i, c] \varpi_i$. It was noted in \cite[Theorem 2.2]{kamnitzerHighestWeightsTruncated2019} that there exist embeddings of crystals
\begin{equation}
	\cB(\wt(\bR)) \injto \cM(\bR) \injto \bigotimes_{(i, c) \in \bR} \cB(\varpi_i)	
\end{equation}
and that furthermore, by varying $\bR$ while keeping $\wt \bR$ fixed, both extremes $\cM(\bR) \cong cB(\wt(\bR))$ and $\cM(\bR) \cong \bigotimes_{(i, c) \in \bR} \cB(\varpi_i)$ can be achieved.

\begin{Example}
	Let $G = \SL_4$, and fix $\wt(\bR) = 2 \varpi_2$. Then depending on $\bR$, there are three possibilities for the isomorphism class of $\cM(\bR)$:
	\begin{itemize}
		\item If $\bR = \{(2, k)^2\}$ for some $k$, then $\cM(\bR) \cong B(2 \varpi_2)$.
		\item If $\bR = \{(2, k), (2, k+2)\}$ for some $k$, then $\cM(\bR) \cong \cB(2 \varpi_2) \oplus \cB(\varpi_1 + \varpi_3)$.
		\item Otherwise, $\cM(\bR) \cong \cB(2 \varpi_2) \oplus \cB(\varpi_1 + \varpi_3) \oplus \cB(0) \cong \cB(\varpi_2)^{\otimes 2}$.
	\end{itemize}
\end{Example}
\section{Analysis of the product monomial crystal} \label{section:analysis-of-product-monomial-crystal}

In this section we give a high-level analysis of the product monomial crystal, which will lay the foundation for the more precise analysis in \cref{section:demazure-truncations-and-character-formula} leading to the character formula. The results of this section appear in \cite{kamnitzerHighestWeightsTruncated2019}, however we go into more detail here.

\subsection{Labelling elements of the product monomial crystal} \label{section:s-labelling}

Let $\bR$ and $\bS$ be finite multisets based in $I \ptimes \bbZ$, and define the auxiliary monomials
\begin{equation}
	y_\bR := \prod_{(i, c) \in \bR} y_{i, c},
	\quad \quad
	z_\bS := \prod_{(i, k) \in \bS} z_{i, k} = \prod_{(i, k) \in \bS} \frac{y_{i, k} y_{i, k + 2}}{\prod_{j \sim i} y_{j, k+1}},
	\quad \quad
	z_\bS^{-1} = \left(z_\bS\right)^{-1}.
\end{equation}

By the definition of the monomial crystal $\cM(G)$, each element $p \in \cM(i, c)^n$ is of the form $p = y_{i, c}^n z_\bS^{-1}$ for some finite multiset $\bS$ based in $I \ptimes \bbZ$. Hence each element $p$ of the product monomial crystal $\cM(\bR)$ is of the form $y_\bR z_\bS^{-1}$ for some finite multiset $\bS$. In fact, by the linear independence of the $z_{i, k}$ in the abelian group $\cM(G)$, the multiset $\bS$ is uniquely determined by $p$. Using this labelling scheme, when $p = y_\bR z_\bS^{-1} \in \cM(\bR)$ the exponent $p[i, k]$ is
\begin{equation}
\label{equation:p-coefficient}
p[i, k] = \bR[i, k] - \bS[i, k - 2] - \bS[i, k] + \sum_{j \sim i} \bS[j, k - 1].
\end{equation}

\begin{Remark}
	In type $A$, the $\bS$ multisets arising in $\cM(\bR)$ may be interpreted in terms of partitions ``hung from pegs'', with each peg corresponding to an element of $\bR$. For more details (and a picture of this), see \cite[Section 2.5.3]{websterQuantumMirkoviCVybornov2017}. While we do not apply this interpretation explicitly in this paper, the author found it invaluable to make the connection between $\cM(\bR)$ and the generalised Schur modules.
\end{Remark}

\subsection{A partial order, upward-closed and downward-closed sets}

We assume from now on that the Dynkin diagram $I$ is connected. Define a partial order $\leq$ on the set $I \ptimes \bbZ$ as the transitive closure of $(i, c) \leq (i, c+2)$ and $(i, c) \leq (j, c+1)$ for all $j \sim i$. 

A subset $J \subseteq I \ptimes \bbZ$ is called \textit{upward-closed} if whenever $x \in J$ and $y \in L$ satisfy $x \leq y$, then $y \in J$. If $J$ is upward-closed, a \textit{minimal element} $x \in J$ is one such that for all $y \in L$, either $x \leq y$ or $x$ and $y$ are incomparable. Every upward-closed set is a union of the upward-closed sets generated by its minimal elements. For any subset $X \subseteq I \ptimes \bbZ$, define $\up(X) = \{y \in I \ptimes \bbZ \mid x \leq y\}$ to be the upward-closed set generated by $X$, and define $\down(X)$ similarly.

For each $i \in I$, define the \textit{$i$-boundary} of an upward-closed set $J$ to be $\partial_i J = (i, k)$, where $(i, k) \in J$ but $(i, k-2) \notin J$. The \textit{boundary} of an upward-closed set $J$ is $\partial J = \{\partial_i J \mid i \in I\}$. A minimal point of $J$ is a boundary point, but boundary points are not necessarily minimal.

\begin{figure}[ht]
	\centering
	\includegraphics[width=0.7\linewidth]{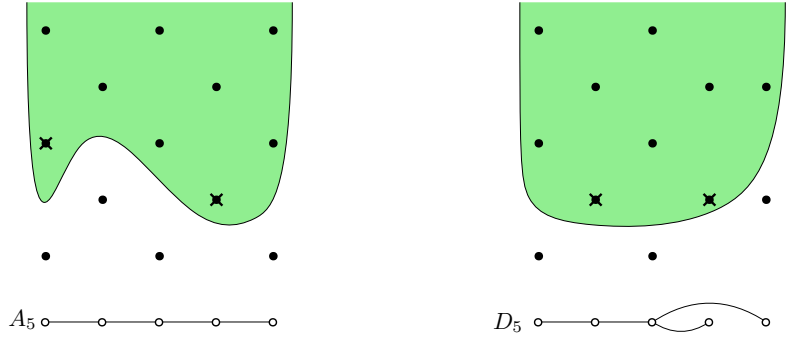}
	\caption{Each diagram depicts an upward-closed set $J \subseteq I \ptimes \bbZ$, as the points enclosed within and above the green region. The points marked with a cross are minimal points of $J$. Both upward-closed sets shown above have five boundary points.}
\end{figure}

A good reason to introduce this order is that the fundamental subcrystals $\cM(i, c)^n$ always ``grow downwards'' from the point $(i, c)$ with respect to this partial order, as made precise in the following lemma.

\begin{Lemma} \label{lemma:fundamental-s-support}
	If $y_{i, c}^n z_\bS^{-1} \in \cM(i, c)^n$, then $\Supp \bS \subseteq \down(\{(i, c - 2)\})$.
\end{Lemma}
\begin{proof}
	The claim is vacuous for the highest-weight element $y_{i, c}^n$ since its associated $\bS$-multiset is empty. As $\cM(i, c)^n$ is connected, it is enough to show that the $f_i$ operators preserve the above property.
	
	Suppose that $p = y_{i, c}^n z_\bS^{-1} \in \cM(i, c)$ satisfies $\Supp \bS \leq (i, c-2)$, and $f_j(p) \neq 0$, so $f_j(p) = p z_{j, k-2}^{-1}$ (i.e. $f_j(p)$ adds the point $(j, k-2)$ to $\bS$). In particular, the definition of $\varphi_j$ gives that $p[j, k] > 0$; this inequality together with \cref{equation:p-coefficient} then implies
	\begin{equation}
		\bR[j, k] + \sum_{l \sim j} \bS[l, k-1] > \bS[j, k] + \bS[j, k-2] \geq 0
	\end{equation}
	where $\bR$ is the multiset $\bR = \{(i, c)^n\}$.
	
	If $\bR[j, k] = n$, then $(j, k) = (i, c)$, and in this case $(j, k - 2) = (i, c - 2)$ and so the property holds. Otherwise, $\bR[j, k] = 0$ and since $\sum_{l \sim j} \bS[l, k-1]$ is strictly positive, at least one upward neighbour of $(j, k-2)$ is already included in $\Supp \bS$, and the claim follows by the transitivity of $\leq$.
\end{proof}

\begin{figure}[ht]
	\centering
	\includegraphics[width=0.7\linewidth]{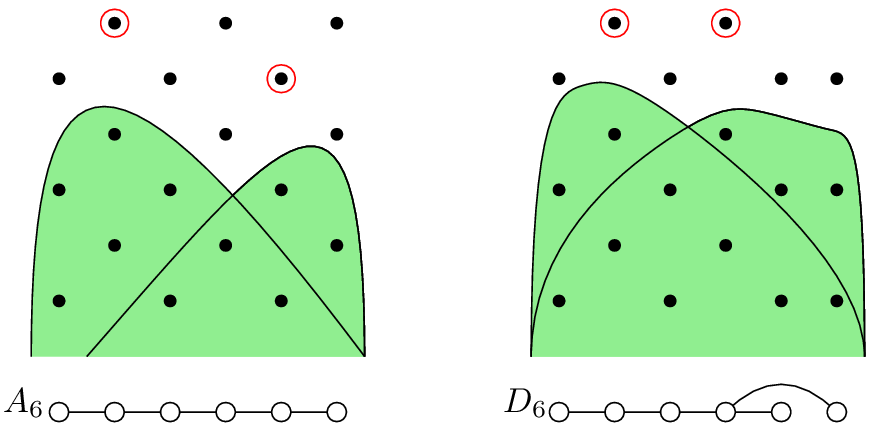}
	\caption{An illustration of \cref{lemma:fundamental-s-support}. Within each picture, a red circled point represents a chosen $(i, c) \in I \ptimes \bbZ$, and the green region directly below shows $\down(\{(i, c-2)\})$. If $y_{i, c}^n z_\bS^{-1} \in \cM(i, c)^n$, then $\Supp \bS \subseteq \down(\{(i, c-2)\})$. Taking the red circled points as defining a multiset $\bR$, then $y_\bR z_\bS^{-1} \in \cM(\bR)$ implies that $\Supp \bS$ is contained in the union of green regions.}
\end{figure}

\subsection{Supports of monomials, highest-weight monomials}

Given a monomial $p = y_\bR z_\bS^{-1} \in \cM(\bR)$, its \textit{$\bR$-support} is defined to be the set $\Supp_\bR(p) = \Supp \bR \cup \Supp \bS$. Note that the $\bR$-support is defined in terms of the $\bS$-labelling, and can be large compared to the support of the original monomial in terms of the $y_{i, c}$, as the next example shows.

\begin{Example}
	When $G = \SL_4$ and $\bR = \{(1, 1), (3, 5)\}$, we have $1 = y_\bR z_\bS^{-1}$ where the multiset $\bS = \{(1, 1), (2, 2), (3, 3)\}$, and so $\Supp_\bR (1) = \{(1, 1), (2, 2), (3, 3), (3, 5)\}$.
\end{Example}

For a monomial $p \in \cM(\bR)$, \cref{equation:p-coefficient} shows that if $p[i, k] \neq 0$, then $(i, k)$ is contained in the upward-closed set generated by $\Supp_\bR (p)$. The next lemma shows that if the $\bR$-support of a monomial $p$ extends below $\bR$ itself, then $p$ cannot be highest-weight.

\begin{Lemma} \label{lemma:ei-deletes-minimal-support}
	Let $p \in \cM(\bR)$, and suppose that $(i, k)$ is minimal in $\Supp_\bR (p)$ and $(i, k) \notin \Supp \bR$. Let $n = \varepsilon_i(p)$ and $q = e_i^n(p)$ be the element at the top of $p$'s $i$-string. Then $n > 0$, and $\Supp_\bR (q) \subseteq \Supp_\bR (p) \setminus (i, k)$.
\end{Lemma}
\begin{proof}
	Let $p = y_\bR z_\bS^{-1}$ and $q = y_\bR z_\bT^{-1}$. Since $\Supp_\bR q \subseteq \Supp_\bR p$, by minimality of $(i, k)$ we have that $p[i, r] = q[i, r] = 0$ for $r < k$, and hence $\varepsilon_i^k(p) = -\bS[i, k]$ and $\varepsilon_i^k(q) = -\bT[i, k]$. Then $n = \varepsilon_i(p) \geq \bS[i, k] > 0$ shows that $n > 0$, and $0 = \varepsilon_i(q) \geq \bT[i, k]$ shows that $(i, k) \notin \Supp_\bR q$.
\end{proof}

\begin{Corollary} \label{corollary:highest-weight-upward-set}
	If $p \in \cM(\bR)$ is highest-weight, then $\Supp_\bR(p) \subseteq \up(\bR)$.
\end{Corollary}
\section{Truncations and a character formula} \label{section:demazure-truncations-and-character-formula}

In this section, we define a family of subsets $\cM(\bR, J)$ of the product monomial crystal $\cM(\bR)$, parametrised by an upward-closed set $J \subseteq I \ptimes \bbZ$. Each subset $\cM(\bR, J)$ is closed under the crystal raising operators $(e_i)_{i \in I}$, and in fact is a Demazure crystal (as defined in \cref{section:demazure-crystals}). We show how to relate the subsets $\cM(\bR, J)$ for varying $\bR$ and $J$, leading to an inductive Demazure-type character formula for each $\cM(\bR, J)$.

\subsection{Truncations defined by upward sets} \label{section:definition-truncation}

We saw in \cref{corollary:highest-weight-upward-set} that the highest-weight elements $p \in \cM(\bR)$ satisfy $\Supp_\bR(p) \subseteq \up(\bR)$. Let $\cM(\bR, \up(\bR))$ be the subset of all monomials in $\cM(\bR)$ whose $\bR$-support is contained in $\up(\bR)$. The subset $\cM(\bR, \up(\bR))$ is closed under the crystal raising operators $e_i$, and contains all highest-weight elements of $\cM(\bR)$. It is the prototypical example of one of our truncations.

\begin{Definition} \label{definition:truncation}
	Let $J \subseteq L$ be an upward-closed set containing $\Supp \bR$. The \defn{truncation of $\cM(\bR)$ by $J$} is the subset
	\begin{equation}
		\cM(\bR, J) = \{p \in \cM(\bR) \mid \Supp_\bR p \subseteq J\} \subseteq \cM(\bR).
	\end{equation}
	We equip this subset with the canonical upper-seminormal abstract crystal structure coming from the restrictions of $\wt$ and $e_i$ to the subset $\cM(\bR, J)$.
\end{Definition}

Some facts about these subsets are already clear:
\begin{enumerate}
	\item Each truncation $\cM(\bR, J)$ contains every highest-weight element of $\cM(\bR)$, by \cref{corollary:highest-weight-upward-set} and the fact that $J$ contains $\Supp \bR$.
	\item Each truncation $\cM(\bR, J)$ is closed under the crystal raising operators $(e_i)_{i \in I}$, since each $e_i$ will either act by zero, or by removing an element from the $\bS$-multiset of a monomial.
	\item If $p \in \cM(\bR, J)$, then the exponent $p[i, c] = 0$ for any $(i, c) \notin J$, by \cref{equation:p-coefficient}.
\end{enumerate}

Suppose we have two upward-closed sets $J, J'$ which both contain $\bR$, and differ by a single element: $J' = J \cup \{(i, c)\}$. Then the two truncations defined by $J$ and $J'$ can be related by purely crystal-theoretic terms: the larger one will be the extension of $i$-strings (see \cref{section:demazure-crystals}) of the smaller one.

\begin{Lemma} \label{lemma:truncation-extension-of-i-strings}
	Let $J$ and $J'$ be two upward-closed sets containing $\Supp \bR$, which differ in a single element: $J' = J \cup \{(i, k)\}$. Then $\cM(\bR, J') = \fD_i \cM(\bR, J)$, where $\fD_i$ is defined in \cref{section:demazure-crystals}.
\end{Lemma}
\begin{proof}
	Suppose that $q \in \cM(\bR, J')$. If $(i, k) \notin \Supp_\bR q$, then $q \in \cM(\bR, J)$. Otherwise, if $(i, k) \in \Supp_\bR q$, then $(i, k)$ is in fact \textit{minimal} in $\Supp_\bR q$, since both $J$ and $J'$ are upward-closed sets. Hence \cref{lemma:ei-deletes-minimal-support} gives that there is some $n > 0$ for which $e_i^n(q) \in \cM(\bR, J)$.
	
	Conversely, suppose that $p \in \fD_i \cM(\bR, J)$. As $(i, k) \notin J$, we have that $p[i, r] = 0$ for all $r \leq k$. It follows from the definition of $f_i$ that both $f_i^n(p)[i, k] \leq 0$ and $f_i^n(p)[i, r] = 0$ for all $n \geq 0$, $r < k$, and hence $\Supp_\bR f_i^n(p) \subseteq J'$.
\end{proof}

The consequence of \cref{lemma:truncation-extension-of-i-strings} is that although the two subsets $\cM(\bR, J)$ and $\cM(\bR, J')$ are defined in terms of supports of monomials, they are related via a purely crystal-theoretic means (the extension of $i$-strings $\fD_i$). However, \cref{lemma:truncation-extension-of-i-strings} can only be used when both the upward-closed sets $J$ and $J'$ include $\Supp \bR$. The next lemma shows how to relate various parameter multisets $\bR$ and $\bR'$, while holding the truncating set $J$ fixed.

\begin{Lemma} \label{lemma:boundary-factorise}
	Suppose that $J$ is an upward-closed set containing $\bR$, and $\bQ$ is a multiset supported along the boundary of $J$: $\Supp \bQ \subseteq \partial J$. Then $\cM(\bR + \bQ, J) = y_\bQ \cdot \cM(\bR, J)$, where $\cdot$ denotes a product as monomials.
\end{Lemma}
\begin{proof}
	Suppose that $p \in \cM(\bR, J)$ and $q \in \cM(\bQ, J)$. By linear independence of the $z_{i, k}$ monomials, we have $\Supp_{\bR + \bQ}(pq) = \Supp_\bR(p) \cup \Supp_\bQ(q)$, which shows that $\cM(\bR, J) \cdot \cM(\bQ, J) = \cM(\bR + \bQ, J)$. (Note that $\cM(\bR) \cdot \cM(\bQ) = \cM(\bR + \bQ)$ by definition of the product monomial crystal). Since $\bQ$ lies along the boundary $\partial J$, \cref{lemma:fundamental-s-support} gives that $\cM(\bQ, J) = \{y_\bQ\}$, and the claim follows.
\end{proof}

\begin{Example} \label{example:compute-truncation}
	We will use the lemmas above to compute the decomposition of the $\SL_4$ crystal $\cM(\bR)$ for $\bR = \{(1, 3), (3, 1), (3, 3)\}$. These steps are shown pictorially in \cref{figure:example:compute-truncation}.

	\begin{enumerate}
		\item Begin with $J_0 =  \up(\{(2, 2)\})$ and $\bR_0 = \varnothing$ to get $\cM(\bR_0, J_0) = \{1\}$.
		\item Let $\bR_1 = \{(1, 3), (3, 3)\}$ and $J_1 = J_0$. Since $\bR_1$ lies along $\partial J_1$ we may apply \cref{lemma:boundary-factorise} to find $\cM(\bR_1, J_1) = y_{1, 3} y_{3, 3} \cdot \cM(\bR_0, J_0) = \{y_{1, 3} y_{3, 3}\}$.
		\item Let $\bR_2 = \bR_1$ and $J_2 = \up\{(3, 1)\}$. Since $J_2$ differs in exactly one element from $J_1$ we apply \cref{lemma:truncation-extension-of-i-strings} to get $\cM(\bR_2, J_2) = \fD_3 \cM(\bR_1, J_1)$. The extension of $3$-strings introduces one more element, giving $\cM(\bR_2, J_2) = \{ y_{1, 3} y_{3, 3},  y_{1, 3} y_{3, 3} z_{3, 1}^{-1}\}$.
		\item Let $\bR_3 = \bR_2 \cup \{(3, 1)\}$ and $J_3 = J_2$. Since $(3, 1)$ lies on the boundary $\partial J_2$, we may apply \cref{lemma:boundary-factorise} once more to find $\cM(\bR_3, J_3) = y_{3, 1} \cdot \cM(\bR_2, J_2) = \{y_{1, 3} y_{3, 3} y_{3, 1}, y_{1, 3} y_{3, 3} y_{3, 1} z_{3, 1}^{-1}\}$.
	\end{enumerate}

	The two elements in $\cM(\bR_3, J_3) = \cM(\bR, J_3)$ are both highest-weight, and so \cref{corollary:highest-weight-upward-set} implies that these are precisely the highest-weight elements of $\cM(\bR)$. Taking their weights gives the decomposition
	\begin{equation}
		\cM(\bR) \cong \cB(\varpi_1 + 2 \varpi_3) \oplus \cB(\varpi_1 + \varpi_2).
	\end{equation}

	\begin{figure}[ht]
		\centering
		\includegraphics[width=0.9\linewidth]{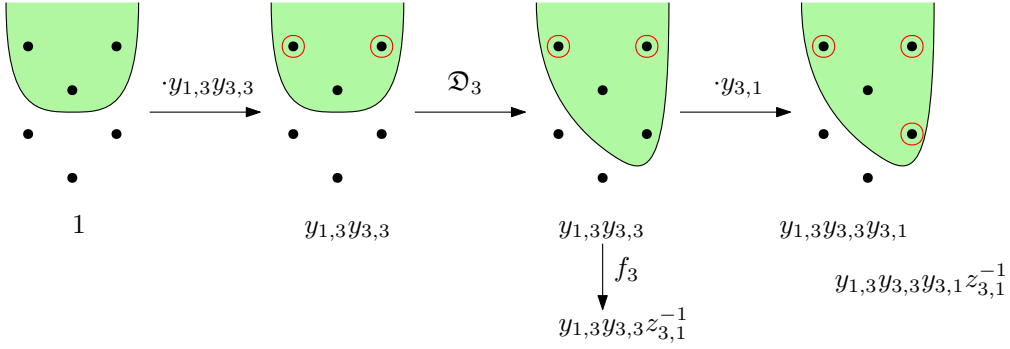}
		\caption{An illustration of \cref{example:compute-truncation} for $\fsl_4$. In each diagram the green shaded region is an upward-closed set $J$, the red circled points are a parameter multiset $\bR$, and below the diagram the truncated crystal $\cM(\bR, J)$ is shown.}
		\label{figure:example:compute-truncation}
	\end{figure}
\end{Example}

\subsection{A Demazure character formula} \label{sec:character-formula}

An \defn{$i$-root string} in an abstract crystal $B$ is a subset of the form $S = \{\ldots, e_i^2(b), e_i(b), b, f_i(b), f_i^2(b), \ldots\}$ for some $b \in B$. If $B$ is seminormal it decomposes into disjoint $i$-root strings, each of finite length. The following definition and theorem are due to Kashiwara \cite{kashiwaraCrystalBaseLittelmann1993}:

\begin{Definition} \label{definition:string-property}
	A subset $X$ of the abstract $\fg$-crystal $B$ has the \defn{string property} if, for every $i$-root string $S$ in $B$, $S \cap X$ is one of $S$, $\varnothing$, or $\{u_S\}$, where $u_S \in S$ satisfies $e_i(u_S) = 0$.
\end{Definition}

\begin{Example}
	Take $B = \cB(\varpi_1 + \varpi_2)$ the connected $\fsl_3$-crystal of the adjoint representation, shown on the left of \cref{figure:string-property}. The subset of $B$ consisting of the six vertices shown in the centre of \cref{figure:string-property} does not have the string property, as the $1$-string of length two violates the conditions of \cref{definition:string-property}. However, the subset of $B$ consisting of the seven vertices shown on the right of \cref{figure:string-property} does have the string property.
	\begin{figure}[ht]
		\centering
		\includegraphics[width=0.8\linewidth]{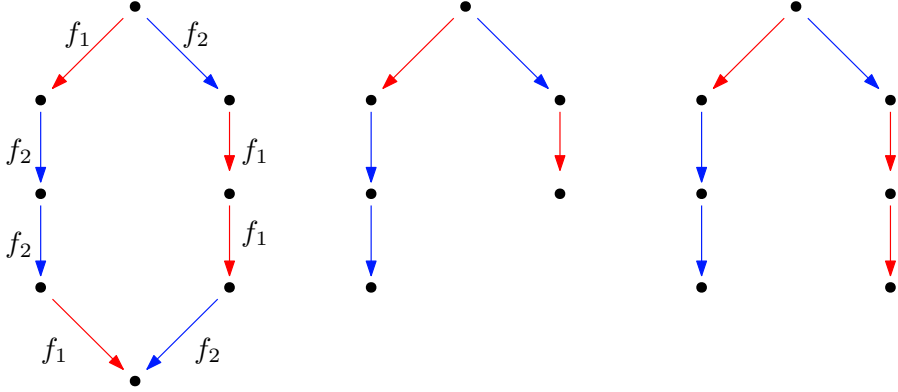}
		\label{figure:string-property}
		\caption{On the left, the connected $\fsl_3$ crystal of highest-weight $\varpi_1 + \varpi_2$. In the middle, a subset without the string property, and on the right a subset  which does have the string property.}
	\end{figure}
\end{Example}

Recall from \cref{section:demazure-crystals} the definition of the Demazure operator $\pi_i$ and the extension-of-strings operator $\fD_i$. The following theorem of Kashiwara establishes a commutation property of the character function with $\pi_i$ and $\fD_i$.

\begin{Theorem} \label{theorem:string-extension-commutes}
	\cite{kashiwaraCrystalBaseLittelmann1993, bumpCrystalBasesRepresentations2017}. If $X$ is a subset of an abstract $\fg$-crystal $B$, and $X$ satisfies the string property, then $\ch(\fD_i X) = \pi_i (\ch X)$ for all $i \in I$.
\end{Theorem}

It is not true that if some subset $X$ has the string property, then $\fD_i(X)$ has the string property --- for a counterexample, see \cite[Chapter 13]{bumpCrystalBasesRepresentations2017}. However, we can verify directly that all of the truncations we have been considering have the string property.

\begin{Lemma}
	\label{lemma:truncations-have-string-property}
	If $J$ is an upward-closed set containing $\bR$, then $\cM(\bR, J)$ has the string property.
\end{Lemma}

\begin{proof}
	Since $\cM(\bR, J)$ is closed under the $e_i$ operators, it suffices to show that for any $p \in \cM(\bR, J)$ with $f_i(p) \notin \cM(\bR, J) \sqcup \{0\}$, that $e_i(p) = 0$. Suppose we have such a $p = y_\bR z_\bS^{-1}$, then by \cref{lemma:truncation-extension-of-i-strings} we must have $f_i(p) = z_{i, k - 2}^{-1} p \in \fD_i \cM(\bR, J)$, where $(i, k) \in \partial J$. By definition of $\varphi_i$, $k$ is largest such that $\varphi_i^k(p) = \varphi_i(p)$, and hence $\varphi_i^{l+2}(p) < \varphi_i^k(p)$ for all $l \geq k$. But since $p[i, r] = 0$ for all $r < k$, we have that $\varepsilon_i^l(p) = \varphi_i^k(p) - \varphi_i^{l + 2}(p) > 0$ for all $l \geq k$, and hence $\varepsilon_i(p) = 0$ and $e_i(p) = 0$.
\end{proof}

We then arrive at an inductive character formula for any truncation $\cM(\bR, J)$ by interpreting both \cref{lemma:truncation-extension-of-i-strings} and \cref{lemma:boundary-factorise} in terms of characters, using \cref{theorem:string-extension-commutes}.

\begin{Theorem} \label{theorem:inductive-character-formula}
	The following rules give an inductive character formula for any truncation $\cM(\bR, J)$:
	\begin{enumerate}
		\item $\ch \cM(\varnothing, J) = 1$ for any $J$.
		\item Suppose $J$ is an upward-closed set containing $\bR$, and $\bQ$ lies along the boundary $\partial J$. Then $\ch \cM(\bR + \bQ, J) = e^{\wt(\bQ)} \cdot \ch \cM(\bR, J)$.
		\item Suppose $J$ is an upward-closed set containing $\bR$, and $J' = J \cup \{(i, k)\}$ is another upward-closed set such that $J \setminus J' = \{(i, k)\}$. Then $\ch \cM(\bR, J') = \pi_i \ch \cM(\bR, J)$.
	\end{enumerate}
\end{Theorem}
\begin{proof}
	Each rule is justified as follows:
	\begin{enumerate}
		\item The truncation $\cM(\varnothing, J)$ is the trivial monomial $\{1\}$, which has character $1$.
		\item If $p \in \cM$ is any monomial and $X \subseteq \cM$ is any subset, then $\ch (p \cdot X) = e^{\wt p} \cdot \ch(X)$. Hence (2) follows from \cref{lemma:boundary-factorise}.
		\item By \cref{lemma:truncation-extension-of-i-strings}, $\cM(\bR, J') = \fD_i \cM(\bR, J)$, and by \cref{lemma:truncations-have-string-property}, $\cM(\bR, J)$ has the string property. The result then follows by taking characters and applying \cref{theorem:string-extension-commutes}.
	\end{enumerate}
\end{proof}

\begin{Example} \label{example:compute-character}
	Let $G = \GL_4$ with the parameter multiset $\bR = \{(1, 3), (3, 1), (3, 3)\}$, similarly to \cref{example:compute-truncation}. We previously determined that when $J = \up(\{(3, 1)\})$ then we had the equality of sets
	\begin{equation}
		\cM(\bR, J) = y_{3, 1} \cdot \fD_3( y_{1, 3} y_{3, 1} \cdot \{1\}).
	\end{equation}
	Applying the rules in \cref{theorem:inductive-character-formula}, we get
	\begin{equation}
		\ch \cM(\bR, J) = e^{\varpi_3} \cdot \pi_3( e^{\varpi_1 + \varpi_3} \cdot 1),
	\end{equation}
	which is most easily computed using the isomorphism $\bbZ[P] \cong \bbZ[x_1, x_2, x_3, x_4]$, choosing our system of fundamental weights to be $\varpi_1 = x_1$, $\varpi_2 = x_1 x_2$ and $\varpi_3 = x_1 x_2 x_3$:
	\begin{align}
		\ch \cM(\bR, J)
		&= x_1 x_2 x_3 \cdot \pi_3(x_1^2 x_2 x_3) \\
		&= x_1^3 x_2^2 x_3 \cdot \pi_3(x_3) \\
		&= x_1^3 x_2^2 x_3 (x_3 + x_4) \\
		&= x_1^3 x_2^2 x_3^2 + x_1^2 x_2 \\
		&= e^{\varpi_1 + 2\varpi_3} + e^{\varpi_1 + \varpi_2}.
	\end{align}
	As $\varpi_1 + 2 \varpi_3$ and $\varpi_1 + \varpi_2$ are dominant, each of $e^{\varpi_1 + 2\varpi_3}$ and $e^{\varpi_1 + \varpi_2}$ is a Demazure character, and hence the Demazure character formula implies
	\begin{equation}
		\cM(\bR) \cong \cB(\varpi_1 + 2 \varpi_3) \oplus \cB(\varpi_1 + \varpi_2).
	\end{equation}
\end{Example}

\subsection{Truncations are Demazure crystals} \label{section:truncations-are-demazure}

The procedure given in \cref{theorem:inductive-character-formula} looks quite similar to the construction of a Demazure crystal (\cref{section:demazure-crystals}), and we will show that each truncation $\cM(\bR, J)$ is in fact a Demazure crystal. We rely on the main result of \cite{josephDecompositionTheoremDemazure2003}, which states that if $X$ is a Demazure crystal and $b$ is a highest-weight element of some crystal, then $\{b\} \otimes X$ is a Demazure crystal. We use this result by formulating \cref{lemma:boundary-factorise} in purely crystal-theoretic terms.

\begin{Lemma} \label{lemma:tensor-embedding-inductive-refined}
	Let $J$ be an upward-closed set containing $\bR$, and let $\bQ$ be a multiset supported along the boundary of $J$, so $\Supp \bQ \subseteq \partial J$. Set $\mu = \wt(\bQ)$, and write $b_\mu \in \cB(\mu)$ for the highest-weight element. There is a bijective, weight-preserving map
	\begin{equation}
		\Phi: \cM(\bR + \bQ, J) \to \cB(\mu) \otimes \cM(\bR, J), \quad \Phi(p) = b_\mu \otimes p / y_\bQ,
	\end{equation}
	which is equivariant under the crystal raising operators, i.e. $\Phi(e_i p) = e_i(\Phi(p))$ for all $i \in I$. Hence $\cM(\bR + \bQ, J) \cong \cB(\mu) \otimes \cM(\bR, J)$ as abstract crystals.
\end{Lemma}
\begin{proof}
	The map is defined as a consequence of \cref{lemma:boundary-factorise}, and is bijective and weight-preserving, so all that remains to be seen is the $e_i$-equivariance. Let us recall the rule (\cref{equation:ei-tensor-first}) for applying $e_i$ to a tensor product of two crystal elements:
	\begin{align}
		e_i(b_\mu \otimes p / y_\bQ)
		&=
		\begin{cases}
			e_i b_\mu \otimes p / y_\bQ & \text{if } \varphi_i(b_\mu) \geq \varepsilon_i(p / y_\bQ) \\
			b_\mu \otimes e_i(p / y_\bQ) & \text{if } \varphi_i(b_\mu) < \varepsilon_i(p / y_\bQ)
		\end{cases} \\
		&= \label{equation:ei-tensor}
		\begin{cases}
			0 & \text{if } \innprod{\alpha_i^\vee, \mu} \geq \varepsilon_i(p / y_\bQ) \\
			b_\mu \otimes e_i(p / y_\bQ) & \text{if } \innprod{\alpha_i^\vee, \mu} < \varepsilon_i(p / y_\bQ)
		\end{cases}
	\end{align}

	Fix an $i \in I$, and let $(i, k) \in \partial J$ be that unique point on the boundary of $J$ lying in column $i$. Let $p \in \cM(\bR + \bQ, J)$ be arbitrary. Since the support of $\bQ$ lies in $\partial J$, the only element of $\bQ$ in column $i$ which could have nonzero multiplicity is $(i, k)$, where it has multiplicity $\bQ[i, k] = \innprod{\alpha_i^\vee, \mu}$. Since $p[i, l] = 0$ for all $l < k$, we then have
	\begin{equation} \label{equation:epsilon-on-boundary}
		\varepsilon_i^l(p)
		=
		\begin{cases}
			0 & \text{for } l < k, \\
			\innprod{\alpha_i^\vee, \mu} + \varepsilon_i^l(p / y_\bQ) & \text{for } l \geq k.
		\end{cases}
	\end{equation}

	If $e_i(p) = 0$, then $\varepsilon_i^l(p) \geq 0$ for all $l$, and hence $\innprod{\alpha_i^\vee, \mu} \geq - \varepsilon_i^l(p/y_\bQ)$ for all $l \geq k$, and hence $\innprod{\alpha_i^\vee, \mu} \geq \max_l -\varepsilon_i^l(p / y_\bQ) = \varepsilon_i(p / y_\bQ)$. Then we are in the first case of \cref{equation:ei-tensor}, and $e_i(\Phi(p)) = 0$.

	On the other hand, if $e_i(p) = z_{i, r} p$ for some $r \geq k$, then $0 < -\varepsilon_i^r(p) = \varepsilon_i(p)$, and applying \cref{equation:epsilon-on-boundary} gives that $\innprod{\alpha_i^\vee, \mu} < -\varepsilon_i^r(p / y_\bQ) \leq \varepsilon_i(p / y_\bQ)$. So we are in the second case of \cref{equation:ei-tensor}, and all that remains to check is that $e_i(p / y_\bQ) = z_{i, r} p / y_\bQ$. However, this is clear from \cref{equation:epsilon-on-boundary} since adding a constant $\innprod{\alpha_i^\vee, \mu}$ to the values of the $\varepsilon_i^\cdot$ will not change the value $r$, as $r$ is defined as the least $l$ such that $\varepsilon_i^l(p)$ is minimised.
\end{proof}

\begin{Theorem} \label{theorem:truncations-are-demazure}
	Let $J$ be an upward-closed set containing $\bR$. Then the truncation $\cM(\bR, J)$ is a Demazure crystal.
\end{Theorem}
\begin{proof}
	We have already shown (\cref{theorem:inductive-character-formula}) that every truncation such as $\cM(\bR, J)$ can be built by starting with the trivial monomial $\{1\}$ and repeatedly applying one of
	\begin{enumerate}
		\item The extension of $i$-string operators $\fD_i$ for $i \in I$; or
		\item Multiplication by a monomial of the form $y_\bQ$, for some finite multiset $\bQ$ based in $I \ptimes \bbZ$.
	\end{enumerate}
	The trivial crystal $\{1\}$ is a Demazure crystal. It follows by \cite[Proposition 3.2.3]{kashiwaraCrystalBaseLittelmann1993} that the property of being a Demazure crystal is preserved under extension of $i$-strings, and by the main theorem of \cite{josephDecompositionTheoremDemazure2003}, the property of being a Demazure crystal is preserved under forming the tensor product with a single highest-weight element $b_\mu$ for any $\mu \in P^+$.
\end{proof}

As a Demazure crystal, the truncation $\cM(\bR, J)$ enjoys several nice properties, such as being determined up to isomorphism by its character $\ch \cM(\bR, J)$.

\begin{Remark}
	It is certainly not true that arbitrary subsets of crystals are determined (up to isomorphism as edge-labelled graphs) by their characters. For example, consider the $\fsl_2$ crystal $B(2) \oplus B(0) = \{a, b, c, d\}$, pictured in \cref{figure:sl2-example}. Both the subsets $X = \{a, d\}$ and $Y = \{a, b\}$ have character $e^{\varpi_1} + e^0$, however $Y$ is not a Demazure crystal while $X$ is.

	\begin{figure}[h]
		\centering
		\includegraphics[width=0.2\linewidth]{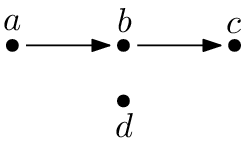}
		\caption{The $\fsl_2$-crystal $B(2) \oplus B(0)$.}
		\label{figure:sl2-example}
	\end{figure}
\end{Remark}

\begin{Corollary} \label{corollary:demazure-stabilisation}
	Since $G$ is a reductive group, $\cM(\bR)$ is a finite crystal and the Weyl group $W$ posesses a unique longest element $w_\circ$. The character of $\cM(\bR)$ may be computed as
	\begin{equation}
		\ch \cM(\bR) = \pi_{w_\circ} \ch \cM(\bR, J)
	\end{equation}
	for any upward-closed set $J$ containing $\bR$.
\end{Corollary}
\begin{proof}
	By the Demazure character formula, a Demazure character is of the form $\pi_x(e^\lambda)$ for some $x \in W$ and $\lambda \in P^+$. As the $\pi_i$ operators braid and are idempotent, they define a $0$-Hecke action on $\bbZ[P]$, in other words a representation of the Hecke algebra associated to $W$ with defining relation $E_i^2 = E_i$ for all $i \in I$. The standard basis element $E_{w_\circ}$ of this Hecke algebra satisfies $E_i E_{w_\circ} = E_{w_\circ} = E_{w_\circ} E_i$ for all $i \in I$, and hence so do the Demazure operators: $\pi_{w_\circ} \pi_x = \pi_{w_\circ}$. We then have that $\pi_{w_\circ} \pi_x(e^\lambda) = \pi_{w_\circ} e^\lambda = \ch V(\lambda)$ by the Demazure character formula.

	Now, $\cM(\bR, J)$ is a disjoint sum $\bigoplus_{\lambda, w} c_{\lambda, w} \cB_w(\lambda)$ of Demazure crystals. Since every highest-weight element of $\cM(\bR)$ appears in the truncation $\cM(\bR, J)$ we have that $\cM(\bR) \cong \bigoplus_{\lambda, w} c_{\lambda, w} \cB(\lambda)$ as $G$-crystals. On the level of characters, this becomes
	\begin{equation}
		\ch \cM(\bR) = \sum_{\lambda, w} c_{\lambda, w} \ch \cB(\lambda) = \pi_{w_\circ} \sum_{\lambda, w} c_{\lambda, w} \pi_w (e^\lambda) = \pi_{w_\circ} \ch \cM(\bR, J).
	\end{equation}
\end{proof}
\section{The product monomial crystal in type A} \label{section:type-a}

We show that in type $A$, the product monomial crystal $\cM(\bR)$ is isomorphic to the crystal of a \textit{generalised Schur module} $\scS_D$, a $\GL_n$-module defined by a \textit{diagram} $D$ depending on $\bR$. The generalised Schur modules have a stable decomposition when $D$ is fixed and $n$ increases, and the character of each is given by a Demazure-type formula \cite{reinerKeyPolynomialsFlagged1995, reinerPercentageAvoidingNorthwestShapes1998}. We first show that the product monomial crystal $\cM(\bR)$ has a stable decomposition when $\bR$ is fixed and $n$ increases, and then compare characters to show that the stable decompositions of $\scS_D$ and $\cM(\bR)$ agree.

\subsection{Stability of decomposition in type A} \label{section:stability}

Throughout this section we will relate crystals of $\GL_n$ for different $n$. In order to fix notation, we give an explicit realisation of the based root datum associated to the reductive algebraic group $\GL_n$.

\begin{Definition} \label{definition:root-datum-gln}
	For $n \geq 1$, the root datum of $\GL_n$ is defined as follows:
	\begin{enumerate}
		\item The indexing set $I_n = \{1, 2, \ldots, n-1\}$,
		\item The cocharacter lattice $P^\vee_n = \bbZ\{\epsilon_1^\vee, \ldots, \epsilon_n^\vee\}$, and cocharacter lattice $P_n = \bbZ\{\epsilon_1, \ldots, \epsilon_n\}$, with pairing $\innprod{\epsilon_i^\vee, \epsilon_j} = \delta_{ij}$,
		\item The simple coroots $\alpha_i^\vee = \epsilon_i^\vee - \epsilon_{i+1}^\vee$,
		\item The simple roots $\alpha_i = \epsilon_i - \epsilon_{i+1}$,
		\item The symmetric map $(i, i) = 2$, $(i, j) = -1$ if $|i - j| = 1$ and $0$ otherwise.
			\[
				(i, j) = \begin{cases}
					2 & \text{if } i = j, \\
					-1 & \text{if } |i - j| = 1, \\
					0 & \text{otherwise}.
				\end{cases}
			\]
	\end{enumerate}
	Furthermore, we fix a system $\varpi_1, \ldots, \varpi_{n-1}$ of fundamental weights, where $\varpi_i = \epsilon_1 + \cdots + \epsilon_i$. Define $\chardet_n = \epsilon_1 + \cdots + \epsilon_n$, then $(\varpi_1, \ldots, \varpi_{n-1}, \chardet_n)$ forms a basis of $P_n$.
\end{Definition}

\begin{Remark}
	For the root datum $\GL_n$, the associated Kac-Moody algebra is $\fgl_n$, and the Weyl group is isomorphic to the symmetric group on $n$ letters, with $s_i: P_n \to P_n$ swapping $\epsilon_i$ and $\epsilon_{i + 1}$ and leaving the other $\epsilon_j$ fixed. The weight $\lambda = \sum_{i = 1}^n \lambda_i \epsilon_i$ is dominant if and only if $\lambda_1 \geq \cdots \geq \lambda_n$, and furthermore we call $\lambda$ a \textit{polynomial} weight if it is dominant and $\lambda_i \geq 0$ for all $1 \leq i \leq n$.
\end{Remark}

Let $I_n = \{1, \ldots, n-1\}$ and $I_\infty = \{1, 2, 3, \ldots\}$. For a finite multiset $\bR$ based in $I_\infty \ptimes \bbZ$, we will say that $\bR$ \textit{lives over $I_n$} if $\Supp \bR \subseteq I_n \ptimes \bbZ$. If $\bR$ lives over both $I_n$ and $I_m$, it defines a $\GL_n$ crystal $\cM(\GL_n, \bR)$ and a $\GL_m$ crystal $\cM(\GL_m, \bR)$. Given two finite multisets $\bR$ and $\bS$ living over $I_n$, we define the Nakajima monomial
\begin{equation}
	v(\GL_n, \bR, \bS)
	= y_\bR z_\bS^{-1}
	= \left( \prod_{(i, c) \in \bR} e^{\varpi_i} y_{i, c} \right)\left( \prod_{(i, k) \in \bS} \frac{\prod_{j \sim_{I_n} i} y_{j, k+1}}{y_{i, k} y_{i, k+2}} \right)
	\in \cM(\GL_n)
\end{equation}
where we have included the full formula to make it clear that the $z_\bS^{-1}$ term depends on $n$. In general we have $v(\GL_n, \bR, \bS) \neq v(\GL_m, \bR, \bS)$ because of this dependence on $n$. However, for $n \leq m$ there is a map of sets from the smaller crystal to the larger one, defined by the $\bS$-parametrisation.

\begin{Lemma}
	\label{lemma:psi-injection-map}
	Suppose that $n \leq m$ and let $\bR$ be a finite multiset living over $I_n$. Then there is an injective map of sets
	\begin{equation}
		\Psi_{n, m}: \cM(\GL_n, \bR) \to \cM(\GL_m, \bR), \quad v(\GL_n, \bR, \bS) \mapsto v(\GL_m, \bR, \bS).
	\end{equation}
	Furthermore, the image of this map is precisely
	\begin{equation}
		\Im \Psi_{n, m} = \{q \in \cM(\GL_m, \bR) \mid \Supp_\bR q \text{ lives over } I_n\}.
	\end{equation}
\end{Lemma}
\begin{proof}
	The claim is trivial for $n = m$, so assume that $n < m$.
	The map is \textit{a priori} an injective map into the monomial crystal $\cM(\GL_m)$, so we need to show that the image is contained in $\cM(\GL_m, \bR)$, and the description of the image is correct.

	From the definition of the monomial crystal $\cB(\GL_n)$, we have for $p = v(\GL_n, \bR, \bS)$ that
	\begin{equation}
		\varphi_i^k(p) = \sum_{l \geq k} \bR[i, l] - \bS[i, l-2] - 2\sum_{l \geq k}\bS[i, l] + \sum_{\substack{l \geq k \\ j \sim_{I_n} i}} \bS[j, l-1].
	\end{equation}
	Only the last summand depends on $n$, and we see that for $q = v(\GL_m, \bR, \bS)$ that
	\begin{equation}
		\varphi_i^k(q) - \varphi_i^k(p)
		= \sum_{\substack{l \geq k \\ j \sim_{I_m} i \\ j \not \sim_{I_n} i}} \bS[j, l-1]
		= \begin{cases}
			0 & \text{if } i \neq n-1 \\
			\sum_{l \geq k} \bS[n, l - 1] = 0 & \text{if } i = n-1
		\end{cases}
	\end{equation}
	Since $\bS$ lives over $I_n$ we must have $\bS[n, -] = 0$ and hence $\varphi_i^k(p) = \varphi_i^k(q)$ for all $i \in I_n$. This implies that if $f_i v(\GL_n, \bR, \bS) = v(\GL_n, \bR, \bS')$ then $f_i v(\GL_m, \bR, \bS) = v(\GL_m, \bR, \bS')$ for all $i \in I_n$.

	Now consider the case of a fundamental subcrystal, where $\bR$ is a single element $\bR = \{(i, c)\}$. Clearly $\Psi_{n, m}$ sends the highest-weight element $y_\bR \in \cM(\GL_n, \bR)$ to $y_\bR \in \cM(\GL_m, \bR)$, and since $\cM(\GL_n, \bR)$ is generated by the $f_i$ for $i \in I_n$, the previous paragraph shows that $\Phi_{n, m}(\cM(\GL_n, \bR)) \subseteq \cM(\GL_m, \bR)$. Furthermore, since the crystal is connected it is clear that the subset of $\cM(\GL_m, \bR)$ of monomials living over $I_n$ is precisely those generated under only $f_i$ for $i \in I_n$, i.e. those monomials $q \in \cM(\GL_m, \bR)$ such that $\Supp_\bR(q)$ lives over $I_n$.

	The claim follows for general $\bR$ by factorisation into a product of monomials coming from various fundamental subcrystals.
\end{proof}

The image of the inclusion map $\Psi_{n, m}$ can be described purely in terms of weights, rather than monomials.
\begin{Lemma}
	\label{lemma:psi-injection-map-weights}
	Let $n \leq m$ and $\bR$ be a finite multiset living over $I_n$. Then
	\begin{equation}
		\Im \Psi_{n, m} = \{q \in \cM(\GL_m, \bR) \mid \innprod{\epsilon_i^\vee, \wt q} = 0 \text{ for all } i > n\}.
	\end{equation}
\end{Lemma}
\begin{proof}
	Let $p = v(\GL_n, \bR, \bS)$. Then since $\wt(p)$ is a linear combination of $\epsilon_1, \ldots, \epsilon_n$ it is certainly true that $\innprod{\epsilon_i^\vee, \wt \Psi_{n, m}(p)} = 0$ for all $i > n$. Conversely, suppose that $q = v(\GL_m, \bR, \bS)$ satisfies $\innprod{\epsilon_i^\vee, \wt q} = 0$ for all $i > n$. Writing $\wt q = \wt \bR - \sum_{i \in I_m} k_i \alpha_i$ for some integers $k_i > 0$, we have
	\begin{equation}
		0 = \innprod{\epsilon_1^\vee + \cdots + \epsilon_m^\vee, \wt q - \wt \bR} = \innprod{\epsilon_1^\vee + \cdots + \epsilon_n^\vee, \wt q - \wt \bR} = - k_n \alpha_n,
	\end{equation}
	showing that $k_n = 0$ and hence $\bS[n, -] = 0$. The same trick can be applied to show that $k_i = 0$ for all $i \geq n$, showing that $\Supp_\bR q$ lives over $I_n$. Hence the claim follows by the second statement of \cref{lemma:psi-injection-map}.
\end{proof}

If $p$ is highest-weight then so is $\Psi_{n, m}(p)$, and therefore the map $\Psi_{n, m}(p)$ restricts to an injection on the highest-weight elements of each crystal:
\begin{equation}
	\psi_{n, m}: \cM(\GL_n, \bR)^\hw \injto \cM(\GL_m, \bR)^\hw,
\end{equation}
where we use the notation $B^\hw$ for the highest-weight elements of the crystal $B$.

\begin{Lemma}
	\label{lemma:inclusions-stabilise}
	Fix a finite multiset $\bR$ based in $I_\infty \ptimes \bbZ$, let $X$ be the intersection of $\up(\bR)$ and $\down(\{(i, c-2 \mid (i, c) \in \bR)\})$, and let $n \geq 1$ be smallest such that $X$ lives over $I_n$. Then:
	\begin{enumerate}
		\item For all $m \geq n$, the inclusion $\psi_{n, m}$ is bijective and weight-equivariant, under the inclusion $P_n \injto P_m$ taking $\epsilon_i$ to $\epsilon_i$.
		\item For $k \leq n$, the image of the inclusion $\psi_{k, n}$ is described purely in terms of weights, by
			\begin{equation}
				\Im \psi_{k, n} = \{p \in \cM(\GL_n, \bR)^\hw \mid \innprod{\epsilon_i^\vee, \wt p} = 0 \text{ for all } i > k\}.
			\end{equation}
	\end{enumerate}
\end{Lemma}
\begin{proof} \mbox{} 
	\begin{enumerate}
		\item If $q \in \cM(\GL_m, \bR)^\hw$, then by both \cref{lemma:fundamental-s-support} and \cref{corollary:highest-weight-upward-set} we have $\Supp_\bR(q) \subseteq X$, and hence the $\psi_{n, m}$ is surjective by the description of $\Im \Psi_{n, m}$ in \cref{lemma:psi-injection-map}. The weight-equivariance follows from the definition of $\Psi_{n, m}$.
		\item Follows from \cref{lemma:psi-injection-map-weights}.
	\end{enumerate}
\end{proof}

The previous lemma shows that the decomposition of $\cM(\GL_n, \bR)$ stabilise when $\bR$ is held fixed and $n$ is allowed to grow, with the smallest $n$ guaranteed to agree with the stable decomposition given in the statement of the lemma.

\begin{Definition} \label{definition:stable-coefficients}
	Let $\bR$ be a finite multiset living over $I_\infty$, and let $n$ be such that $\psi_{n, m}$ is bijective for all $m \geq n$. Define the decomposition multiplicities $c_\bR^\lambda \in \bbN$ by the equation
	\begin{equation}
		\cM(\GL_n, \bR) \cong \bigoplus_{\lambda} \cB(\GL_n, \lambda)^{\oplus c_\bR^\lambda},
	\end{equation}
	where the sum is over all weights $\lambda \in P^+_n$.
\end{Definition}

The fact that the constants $c_\bR^\lambda$ are well-defined is a consequence of the fist part of \cref{lemma:inclusions-stabilise}. The second part of \cref{lemma:inclusions-stabilise} gives us the following rule for using the coefficients $c_\bR^\lambda$ to decompose $\cM(\GL_n, \bR)$ when $n$ is not stable for $\bR$.

\begin{Corollary} \label{corollary:monomial-restriction-rule}
	The decomposition of $\cM(\GL_n, \bR)$ for any $n \geq 1$ is given in terms of the $c_\bR^\lambda$ as
	\begin{equation}
		\cM(\GL_n, \bR) \cong \bigoplus_{\ell(\lambda) \leq n} \cB(\GL_n, \lambda)^{\oplus c_\bR^\lambda},
	\end{equation}
	where only partitions with length at most $n$ appear in the sum.
\end{Corollary}

Recall that a weight $\lambda = \lambda_1 \epsilon_1 + \cdots + \lambda_n \epsilon_n$ of $\GL_n$ is called \defn{polynomial} if $\lambda_i \geq 0$ for all $1 \leq i \leq n$, and is both dominant and polynomial if and only if $\lambda_1 \geq \cdots \geq \lambda_n \geq 0$, or in other words if $(\lambda_1, \ldots, \lambda_n)$ is a \defn{partition} with at most $n$ parts. Let $\Part_n$ denote the set of partitions with at most $n$ parts, $\Part = \bigcup_{n \geq 0} \Part_n$ denote the set of all partitions. We briefly remind the reader how to go between partitions and weights of $\GL_n$.

A partition $\lambda = (\lambda_1 \geq \cdots \geq \lambda_k > 0)$ is a weakly decreasing list of positive integers, and the \textit{length} of the partition $\lambda$ is $\ell(\lambda) = k$, the length of the list. We draw partitions as Young diagrams using English notation, so that for example $\lambda = (4, 3, 1, 1)$ is represented as the diagram
\begin{equation}
	\lambda = \ydiagram{4, 3, 1, 1}.
\end{equation}
A partition $\lambda$ of length at most $n$ may be interpreted as a dominant weight of $\GL_n$, by taking $\langle \lambda, \alpha_i^\vee \rangle$ to be the number of columns of length $i$ for each $1 \leq i \leq n - 1$, and taking the number of columns of length $n$ to be the multiplicity of the determinant. For example, the same partition $\lambda = (4, 3, 1, 1)$ would represent the weight $\varpi_1 + 2 \varpi_2 + \chardet$ of $\GL_4$.

\begin{Example}
	We give a worked example of starting from a finite multiset $\bR$, determining the stable coefficients $c_\bR^\lambda$, and specialising those stable coefficients to any $\GL_n$.
	
	Let $\bR = \{(1, 5), (3, 1), (4, 6)\}$. The set $\up(\bR) \cap \down(\{(i, c-2) \mid (i, c) \in \bR\})$ are depicted as the green shaded regions in \cref{figure:example-compute-character-from-stable}, showing the smallest stable $n$ for this particular multiset is $n = 6$.

	Using a computer, we determine the decomposition of $\cM(\GL_6, \bR)$ to be
	\begin{equation}
		\begin{aligned}
			\cM(\GL_6, \bR)
			&\cong
			\cB(2 \varpi_4) \oplus \cB(\varpi_3 + \varpi_5) \oplus \cB(\varpi_2 + \chardet_6) \oplus \\
			&
			\cB(\varpi_1 + \varpi_3 + \varpi_4) \oplus \cB(\varpi_1 + \varpi_2 + \varpi_5) \oplus \cB(2 \varpi_1 + \chardet_6).
		\end{aligned}
	\end{equation}
	In terms of partitions, the $\lambda$ for which $c_\bR^\lambda = 1$ are
	\begin{equation}
		\ydiagram{2, 2, 2, 2}
		\quad
		\ydiagram{2, 2, 2, 1, 1}
		\quad
		\ydiagram{2, 2, 1, 1, 1, 1}
		\quad
		\ydiagram{3, 2, 2, 1}
		\quad
		\ydiagram{3, 2, 1, 1, 1}
		\quad
		\ydiagram{3, 1, 1, 1, 1, 1},
	\end{equation}
	and $c_\bR^\lambda = 0$ for all other partitions $\lambda$.

	Since $\bR$ lives over $I_5$ it makes sense to specialise to $\cM(\GL_5, \bR)$, whose decomposition we can compute from the stable coefficients $c_\bR^\lambda$ by \cref{corollary:monomial-restriction-rule}: we simply need to throw away the two partitions whose length is more than $5$. In terms of fundamental weights, we obtain the decomposition
	\begin{equation}
		\cM(\GL_5, \bR) \cong \cB(2 \varpi_4) \oplus \cB(\varpi_3 + \chardet_5) \oplus \cB(\varpi_1 + \varpi_3 + \varpi_4) \oplus \cB(\varpi_1 + \varpi_2 + \chardet_5).
	\end{equation}

	\begin{figure}[ht]
		\centering
		\includegraphics[width=0.7\linewidth]{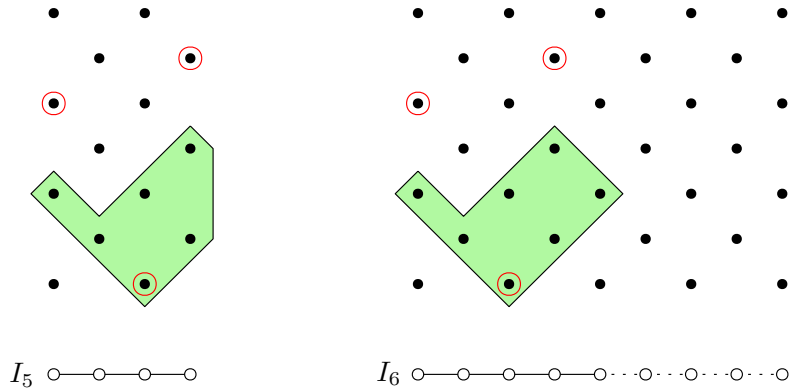}
		\caption{An illustration of the sets $X \cap I_5 \ptimes \bbZ$ and $X$ of \cref{lemma:inclusions-stabilise} for the multiset $\bR = \{(1, 5), (3, 1), (4, 6)\}$. The smallest stable $n$ is $n = 6$.}
		\label{figure:example-compute-character-from-stable}
	\end{figure}
\end{Example}

\begin{Remark}
	Most of the above discussion of stability applies to the family $D_n$ for $n \geq 4$, with the Dynkin diagram labelled as \[ \includegraphics[width=0.3\linewidth]{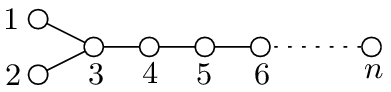} \]
	However, we do not know the correct analogue of generalised Schur modules in type $D_n$, so we have not treated this case here explicitly.
\end{Remark}

\subsection{Specht modules associated to arbitrary diagrams} \label{section:definition-generalised-specht}

Let $D \subseteq \bbN \times \bbN$ be a finite subset with cardinality $d$, which we call a \textit{diagram} with $d$ boxes. A bijection $T: D \to \{1, \ldots, d\}$ is called a \textit{tableau} of $D$, and the symmetric group $\fS_d$ acts on the set of tableaux by postcomposition: $\sigma \cdot T = \sigma \circ T$. A choice of tableau $T$ defines two subgroups of the symmetric group $\fS_d$, the \textit{row stabilising} subgroup $R_T$ which permute the entries of $T$ within their rows, and the \textit{column stabilising} subgroup $C_T$ which permute the entries of $T$ within their columns. Using these two subgroups we may define the \textit{Young symmetriser}  $y_T = \sum_{r \in R_T, c \in C_T} (-1)^r rc$, a pseudo-idempotent element of the group algebra $\bbC[\fS_d]$. The left submodule $\Sigma_T = \bbC[\fS_D] y_T$ is called the \textit{generalised Specht module} associated to the diagram $D$ and tableau $T$. A different choice of $T$ gives an isomorphic module, and so we define $\Sigma_D \cong \Sigma_T$ for any tableau $T$ of $D$, noting $\Sigma_D$ is only defined up to isomorphism.

When $D$ is the Young diagram corresponding to a partition $\lambda$ of $d$, $\Sigma_\lambda$ is called a \textit{Specht module}, and is irreducible. Furthermore, as $\lambda$ runs through the partitions of $d$, the $\Sigma_\lambda$ give a complete set of irreducible modules for $\bbC[\fS_d]$. The decomposition numbers $c_D^\lambda := [\Sigma_\lambda : \Sigma_D]$ are called \textit{generalised Littlewood-Richardson coefficients}, and are invariant under row and column permutations of the diagram $D$.

\begin{Example} \label{example:diagram-becomes-skew}
	Let $D = \{(1, 1), (2, 2), (3, 2), (2, 3), (4, 3)\} \subseteq \bbZ \times \bbZ$. This diagram, along with a tableau $T: D \to \{1, 2, 3, 4, 5\}$ is pictured in \cref{figure:diagram-example}. The row-stabilising subgroup $R_T$ is generated by the transposition $(24)$, while the column-stabilising subgroup $C_T$ is generated by $(23)$ and $(45)$. The diagram $D$ may be made into a skew diagram $D'$ by applying the permutation $(234)$ to the rows, followed by $(132)$ to the columns. We then use the theory of skew Schur functions to compute that $\Sigma_D \cong \Sigma_{D'} \cong \Sigma_{(2, 1, 1, 1)} \oplus \Sigma_{(3, 1, 1)} \oplus \Sigma_{(2, 2, 1)}^{\oplus 2} \oplus \Sigma_{(3, 2)}$. For example, the generalised Littlewood-Richardson coefficient $c_D^{(2, 2, 1)} = 2$.

	\begin{figure}[ht]
		\centering
		\includegraphics[width=0.7\linewidth]{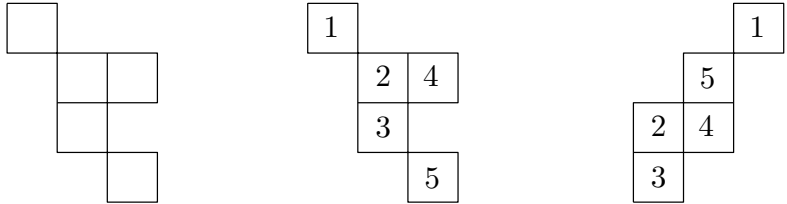}
		\caption{On the left is a diagram $D = \{(1, 1), (2, 2), (3, 2), (2, 3), (4, 3)\}$, pictured as a collection of squares in the plane. We index positions in the plane like matrices, so that the first coordinate goes down the page, and the second coordinate goes right along the page. In the middle is a tableau $T: D \to \{1, \ldots, 5\}$, and on the right is a rearrangement of the rows and columns of $D$ to a second diagram $D'$ which is skew, of shape $(3, 2, 2, 1) / (2, 1)$.}
		\label{figure:diagram-example}
	\end{figure}
\end{Example}

\begin{Remark} \label{remark:non-skew}
	Not all diagrams can be made into a skew shape via row and column permutations, even if we restrict to only row and column permutations. One such example is the diagram
	\[D = \begin{ytableau}
		\none & &      & \none & \none \\
		\none & & \none &       & \none \\
		\none & & \none & \none &       \\
	\end{ytableau}\]
	which has the decomposition $\Sigma_D \cong \Sigma_{(4, 1, 1)} \oplus \Sigma_{(3, 2, 1)}^{\oplus 2} \oplus \Sigma_{(2, 2, 2)}$. (This decomposition is computed in \cref{example:partition-sequence-non-skew}).
\end{Remark}

\subsection{Schur modules associated to arbitrary diagrams} \label{section:definition-generalised-schur}

Let $V$ be a finite-dimensional $\bbC$-vector space, and $D$ be a diagram with $d$ boxes. The tensor power $V^{\otimes d}$ is naturally a $(\GL(V), \fS_d)$ bimodule, and we define the \textit{generalised Schur module} to be (up to isomorphism) the left $\GL(V)$-module $\scS_D(V) \cong V^{\otimes d} \otimes_{\bbC[\fS_d]} \Sigma_D$. As a consequence of Schur-Weyl duality, the generalised Schur module decomposes as $\scS_D(V) \cong \bigoplus_{\ell(\lambda) \leq \dim V} c_D^\lambda \scS_\lambda(V)$. The restriction on partitions having length at most $\dim V$ is not strictly necessary, since in this case we would have $\scS_\lambda(V) = 0$.

The main theorem to be shown is that in type $A$, the product monomial crystal is always a crystal of a generalised Schur module. 

\subsection{Schur and Flagged Schur modules}

In order to study the characters of the Schur module $\scS_D(V)$, it is convenient to introduce a more concrete definition of the Schur module (note that our previous definition was only up to isomorphism), as well as a quotient of the Schur module, called the \textit{flagged Schur module}. While the Schur module is a module for the whole of $\GL(V)$, the flagged Schur module will be a module for a Borel subgroup of $\GL(V)$, and will only be defined when $\dim V$ is large enough compared to $D$. In the following discussion, we adopt the definitions from \cite{reinerFlaggedWeylModules1999}.

For a diagram $D$, let $\operatorname{col}_j(D) \subseteq D$ denote the subset of boxes in column $j$, and $\operatorname{row}_i(D) \subseteq D$ denote the subset of boxes in row $i$. Let $\bigwedge^k(V)$, $T^k(V)$, and $S^k(V)$ denote the exterior, tensor, and symmetric algebras of degree $k$ of $V$. We define the map $\psi_D$ as the composition
\begin{equation}
	\bigotimes_j \bigwedge^{\operatorname{col}_j(D)}(V) \xto{\Delta \otimes \cdots \otimes \Delta}
	\bigotimes_j T^{\operatorname{col}_j(D)}(V) \to
	\bigotimes_i T^{\operatorname{row}_i(D)}(V) \xto{m \otimes \cdots \otimes m}
	\bigotimes_i S^{\operatorname{row}_i(D)}(V),
\end{equation}
where the first map is comultiplication in each exterior algebra, the second map is the natural rearrangement, and the third map is multiplication in the symmetric algebra. The Schur module $\scS_D(V)$ is defined as the $\GL(V)$-submodule $\Im \psi_D$. Note that rearranging the columns of $D$ leaves $\scS_D(V)$ invariant, while rearranging the rows of $D$ yields a different (but isomorphic) Schur module.

Now, suppose that the diagram $D$ satisfies $D \subseteq \{1, \ldots, r\} \times \bbN$ (the diagram fits within rows $1$ through $r$), and $n = \dim V \geq r$. Fix a full flag of quotient spaces $V_\bullet = (V_n \to V_{n-1} \to \cdots \to V_1)$, by which we mean that $\dim V_i = i$ and each map $V_i \to V_{i-1}$ is surjective. We may postcompose $\psi_D$ with the projection
\begin{equation}
	\phi_D: \bigotimes_i S^{\operatorname{row}_i(D)}(V) \to \bigotimes_i S^{\operatorname{row}_i(D)}(V_i),
\end{equation}
and define the \textit{flagged Schur module} $\scS_D^\flag(V_\bullet) = \Im (\phi_D \circ \psi_D)$. Let $B(V_\bullet) \subseteq \GL(V)$ is the subgroup fixing the flag of quotients, then $\scS_D^\flag(V_\bullet)$ will be a $B(V_\bullet)$ module, but rarely a $\GL(V)$ module. Note that again, the module $\scS_D^\flag(V_\bullet)$ is unchanged under column permutations of $D$, but is no longer invariant under row permutations. This construction (and its dual, the Weyl and flagged Weyl modules) are given in full detail in \cite{reinerFlaggedWeylModules1999}, sections 2 and 5.

\begin{Example}
	Let $\lambda$ be a partition with at most $r$ rows, and $D$ be its Young diagram, placed so that the longest row of $\lambda$ is in the first row of $D$. Then $\scS_D(\bbC^r) \cong V(\lambda)$, the irreducible $\GL_r$-module of highest weight $\lambda$, and $\scS_D^\flag(\bbC^r \to \bbC^{r-1} \to \cdots \to \bbC^1) \cong V(\lambda)_\lambda$, the highest-weight space. If the diagram is placed upside-down, so that the longest row of $\lambda$ is in row $r$, then both the Schur and flagged Schur modules are isomorphic to $V(\lambda)$.
\end{Example}

Many of the results known about the generalised Schur modules $\scS_D(V)$ are due to geometric constructions of this module as sections of a line bundle over a (generally singular) variety, in \cite{magyarBorelWeilTheorem1998, magyarSchubertPolynomialsBottSamelson1998}. In this setting, the flagged Schur module (or its dual, the flagged Weyl module) naturally arise, and in \cite{reinerPercentageAvoidingNorthwestShapes1998, reinerKeyPolynomialsFlagged1995}, a Demazure-type character formula is given for the characters of the flagged Schur modules of \textit{percentage-avoiding} diagrams $D$. Fortunately, the diagrams we will encounter are \textit{northwest}, which are automatically percentage-avoiding, and so these results apply. (A diagram $D$ is \textit{northwest} if whenever $(j, k), (i, l) \in D$ with $i<j$ and $k<l$, then $(i, k) \in D$).

\subsection{Diagrams and multisets defined by partition sequences}

It is quite awkward to directly state the map from a multiset $\bR$ to a corresponding diagram $D$. Instead, we will define each of $\bR$ and $D$ from a common partition sequence.

\begin{Definition} \label{definition:partition-sequence}
	A \textit{partition sequence of length $r$} is a sequence $\underline{\lambda} = (\lambda^{(1)}, \ldots, \lambda^{(r)})$ of partitions, such that $\ell(\lambda^{(i)}) \leq i$. For each $0 \leq i \leq r$, let $\underline{\lambda}^i = (\lambda^{(1)}, \ldots, \lambda^{(i)})$ denote the prefix of $\underline{\lambda}$ of length $i$.
\end{Definition}

\begin{Definition}
	Let $\underline{\lambda}$ be a partition sequence of length $r$. The \textit{diagram $D(\underline{\lambda})$ associated to $\underline{\lambda}$} is defined inductively as follows:
	\begin{enumerate}
		\item For $i = 0$, $D(\underline{\lambda}^0) = \varnothing$, the empty diagram.
		\item For $i > 0$, $D(\underline{\lambda}^i)$ is obtained from $D(\underline{\lambda}^{i-1})$ by shifting the contents of $D(\underline{\lambda}^{i-1})$ down one row, and placing the Young diagram of $\lambda^{(i)}$ to the right of the previous diagram, with the longest row of $\lambda^{(i)}$ in row 1.
	\end{enumerate}
\end{Definition}

\begin{Example} \label{example:partition-sequence-to-diagram}
	Given the partition sequence $\underline{\lambda} = (\varnothing, (1, 1), (2, 1), (1, 1, 1, 1), (2, 1, 1))$, we obtain the sequence of diagrams $D(\underline{\lambda}^0) = D(\underline{\lambda}^1) = \varnothing$, then $D(\underline{\lambda}^2), \ldots, D(\underline{\lambda}^5)$ are given by 

\[
\begin{aligned}
D(\underline{\lambda}^2)
=&
\begin{ytableau}
		\none[1] & \none &  \\
		\none[2] & \none &   
\end{ytableau}
&
D(\underline{\lambda}^3)
=&
\begin{ytableau}
		\none[1] & \none & \none   &          &        \\
		\none[2] & \none &  &          & \none  \\
		\none[3] & \none &  & \none    & \none   
\end{ytableau}
\\
\\
D(\underline{\lambda}^4)
=&
\begin{ytableau}
		\none[1] & \none & \none   & \none    & \none    &  \\
		\none[2] & \none & \none   &          &          & \\
		\none[3] & \none &  &          & \none    & \\
		\none[4] & \none &  & \none    & \none    &  
\end{ytableau}
&
D(\underline{\lambda}^5)
=&
\begin{ytableau}
		\none[1] & \none & \none   & \none    & \none    & \none   &       &      \\
		\none[2] & \none & \none   & \none    & \none    &  &       & \none \\
		\none[3] & \none & \none   &          &          &  &       & \none \\
		\none[4] & \none &  &          & \none    &  & \none\\
		\none[5] & \none &  & \none    & \none    &  & \none
\end{ytableau}
\end{aligned}
\]
\end{Example}

\begin{Remark} \label{remark:column-convex}
	All of the results of \cite{reinerPercentageAvoidingNorthwestShapes1998} apply to the class of \textit{column-convex diagrams}, which are diagrams where the columns have no gaps: if $(i_1, j) \in D$ and $(i_2, j) \in D$ for $i_1 < i_2$, then all of $(i_1, j), (i_1 + 1, j), \ldots, (i_2, j) \in D$. We note that diagrams of the form $D(\underline{\lambda})$ are always column-convex, and conversely that every column-convex diagram $D$ is of the form $D(\underline{\lambda})$ after applying a column permutation.
\end{Remark}

\begin{Lemma} \label{lemma:character-of-partition-sequence-diagram}
	Let $\underline{\lambda}$ be a partition sequence of length $r$, and let $V$ be a vector space of dimension $n \geq r$, with a fixed full quotient flag $V_\bullet$. Then the characters of the flagged Schur modules $\scS_{D(\underline{\lambda}^i)}^\flag(V_\bullet)$ satisfy the following recurrence:
	\begin{enumerate}
		\item For $i = 0$, $\ch \scS_{D(\underline{\lambda}^0)}^\flag(V_\bullet) = 1$.
		\item For $i > 0$, $\ch \scS_{D(\underline{\lambda}^i)}^\flag(V_\bullet) = e^{\lambda^{(i)}} \cdot \pi_1 \cdots \pi_{i-1} (\ch \scS_{D(\underline{\lambda}^{i-1})}^\flag(V_\bullet))$ as $\GL(V)$ characters.
	\end{enumerate}
\end{Lemma}
\begin{proof}
	The case of $i = 0$ is clear. The inductive case follows from \cite[Theorem 23]{reinerPercentageAvoidingNorthwestShapes1998}, noting moving the diagram $D(\underline{\lambda}^{i-1})$ down one row can be done by the successive row permutations $s_{(i-1, i)}, \ldots, s_{(1, 2)}$, corresponding to the application of Demazure operators $\pi_{i-1}, \ldots, \pi_1$.
\end{proof}

\begin{Example} \label{example:partition-sequence-non-skew}
	Recall the diagram $D$ from \cref{remark:non-skew}:
	\[D = \begin{ytableau}
		\none & &      & \none & \none \\
		\none & & \none &       & \none \\
		\none & & \none & \none &       \\
	\end{ytableau}\]
	which (after sorting columns) corresponds to the partition sequence $\underline{\lambda} = ((1), (1), (2, 1, 1))$. We get the sequence of diagrams (where we have expanded out the row-swapping steps)
	\[
		\begin{ytableau}
			\none[1] & &      & \none & \none \\
			\none[2] & & \none &       & \none \\
			\none[3] & & \none & \none &       \\
		\end{ytableau}
		\to
		\begin{ytableau}
			\none & \none & \none \\
			\none &       & \none \\
			\none & \none &       \\
		\end{ytableau}
		\to
		\begin{ytableau}
			\none &       & \none \\
			\none & \none & \none \\
			\none & \none &       \\
		\end{ytableau}
		\to
		\begin{ytableau}
			\none &       & \none \\
			\none & \none & \\
			\none & \none \\
		\end{ytableau}
		\to
		\begin{ytableau}
			\none & \none \\
			\none &  \\
			\none \\
		\end{ytableau}
		\to
		\begin{ytableau}
			\none & \\
			\none \\
			\none \\
		\end{ytableau}
		\to
		\varnothing
	\]
	which gives the character formula
	\begin{equation}
		\ch \scS_D^\flag(\bbC^3) = e^{(2, 1, 1)} \cdot \pi_1 \pi_2 (e^{(1)} \cdot \pi_1 (e^{(1)} \cdot 1))
	\end{equation}
	which we compute to be (writing $x_1 \cdots x_i = e^{\varpi_i}$, and using the shorthand $x_1^a x_2^b x_3^c = x^{abc}$)
	\begin{equation}
		\ch \scS_D^\flag(\bbC^3) = x^{411} + x^{231} + 2x^{321} + x^{312} + x^{222}
	\end{equation}
	which decomposes as a sum of Demazure characters
	\begin{equation}
		\ch \scS_D^\flag(\bbC^3) = \kappa^{411} + \kappa^{231} + \kappa^{312} + \kappa^{222}
	\end{equation}
	showing the claimed decomposition $\Sigma_D \cong \Sigma_{(4, 1, 1)} \oplus \Sigma_{(3, 2, 1)}^{\oplus 2} \oplus \Sigma_{(2, 2, 2)}$.
\end{Example}

\begin{Definition} \label{definition:partition-sequence-to-multiset}
	Let $\underline{\lambda}$ be a partition sequence of length $r$. Define the \textit{multiset $\bR(\underline{\lambda})$ associated to $\underline{\lambda}$} and the \textit{upward-closed set $J(\underline{\lambda})$ associated to $\lambda$} inductively as follows:
	\begin{enumerate}
		\item For $i = 0$, $\bR(\underline{\lambda}^0) = \varnothing$, and let $J(\underline{\lambda}^0)$ be the complement of the downward-closed set generated by $(1, -1)$.
		\item For $i > 0$, let $J(\underline{\lambda}^i)$ be the union of $J(\underline{\lambda}^{i-1})$ with the upward-closed set generated by $(1, -2i + 1)$, and let $\bR(\underline{\lambda}^i) - \bR(\underline{\lambda}^{i - 1})$ be supported on $J(\underline{\lambda}^i) \setminus J(\underline{\lambda}^{i-1})$ and have weight $\lambda^{(i)}$.
	\end{enumerate}
\end{Definition}

\begin{Example} \label{example:partitions-sequence-to-multiset}
	As in \cref{example:partition-sequence-to-diagram}, let $\underline{\lambda} = (\varnothing, (1, 1), (2, 1), (1, 1, 1, 1), (2, 1, 1))$. As a sequence of weights, expressed in terms of the fundamental weights, we would instead have $(0, \varpi_2, \varpi_1 + \varpi_2, \varpi_4, \varpi_1 + \varpi_3)$. The multiset $\bR(\underline{\lambda}^5)$ is shown in \cref{figure:partition-sequence-to-multiset} as the red circled points, with the differences in the truncations $J_i = J(\underline{\lambda}^i)$ also shown.
	\begin{figure}[ht]
		\centering
		\includegraphics[width=0.7\linewidth]{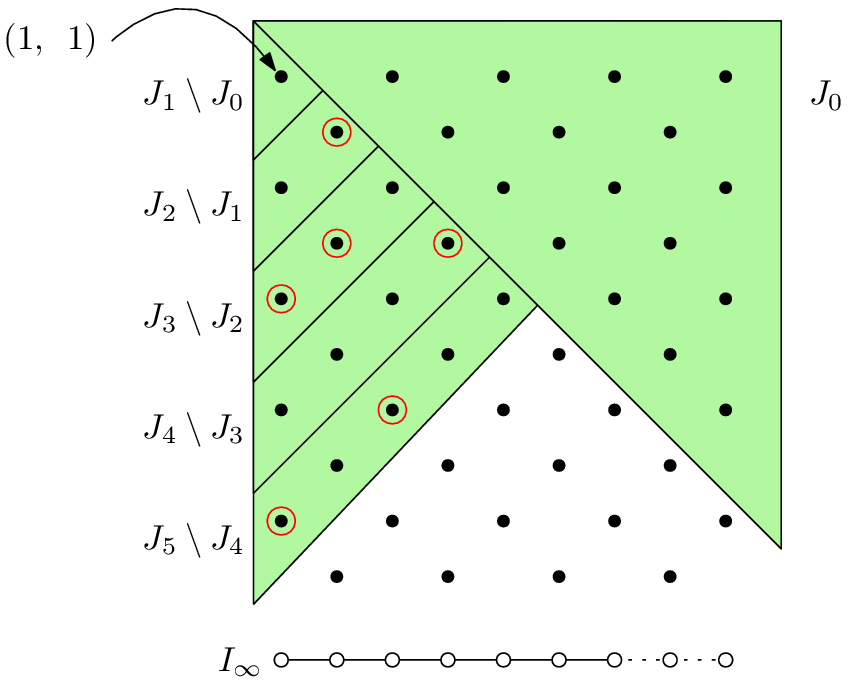}
		\caption{An illustration of \cref{example:partitions-sequence-to-multiset}.}
		\label{figure:partition-sequence-to-multiset}
	\end{figure}
\end{Example}

\begin{Remark}
	If $\underline{\lambda}$ is a partition sequence of length $r$, then $\bR(\underline{\lambda})$ is stable for $\GL_r$ and hence $\cM(\GL_r, \underline{\lambda})$ may be used to compute the stable coefficients $c_\bR^\lambda$ of \cref{definition:stable-coefficients}.
\end{Remark}

\begin{Lemma} \label{lemma:character-of-partition-sequence-multiset}
	Let $\underline{\lambda}$ be a partition sequence of length $r$. Then the characters of the truncated product monomial crystals $\cM(\GL_r, \bR(\underline{\lambda^i}), J(\underline{\lambda}^i))$ satisfy the following recurrence:
	\begin{enumerate}
		\item For $i = 0$, $\ch \cM(\bR(\underline{\lambda}^0), J(\underline{\lambda}^0)) = \ch \cM(\varnothing, J(\underline{\lambda}^0)) = 1$.
		\item For $i > 0$,
		\begin{align}
			\ch \cM(\bR(\underline{\lambda}^i), J(\underline{\lambda}^i))
			&= e^{\lambda^{(i)}} \cdot \ch \cM(\bR(\underline{\lambda}^{i-1}), J(\underline{\lambda}^i)) \\
			&= e^{\lambda^{(i)}} \cdot \pi_1 \cdots \pi_{i-1} \ch \cM(\bR(\underline{\lambda}^{i-1}), J(\underline{\lambda}^{i-1})).
		\end{align}
	\end{enumerate}
\end{Lemma}
\begin{proof}
	The case of $i = 0$ is clear. The inductive character formula given in \cref{theorem:inductive-character-formula} implies the inductive case.
\end{proof}

The recurrences appearing in \cref{lemma:character-of-partition-sequence-diagram} and \cref{lemma:character-of-partition-sequence-multiset} are identical, showing that for a partition sequence $\underline{\lambda}$ of length $r$, the character of the flagged Schur module $\scS_{D(\underline{\lambda})^\flag(V_\bullet)}$ is equal to the character of the truncated crystal $\cM(\GL_r, \bR(\underline{\lambda}), J(\underline{\lambda}))$.

\begin{Theorem}
	\label{theorem:crystal-of-schur-module}
	Let $\underline{\lambda}$ be a parittion sequence of length $r$, and let $\bR = \bR(\underline{\lambda})$ and $D = D(\underline{\lambda})$. Then the stable coefficients of $\bR$ and $D$ coincide, i.e. we have $c_\bR^\mu = c_D^\mu$ for all partitions $\mu$, and furthermore for any $n$ such that $\bR(\underline{\lambda})$ lives over $I_n$ (equivalently, $D(\underline{\lambda})$ has columns of length at most $n$), we have $\ch \cM(\GL_r, \bR) = \ch \scS_D(\bbC^n)$ and hence the monomial crystal $\cM(\GL_r, \bR)$ is the crystal of the generalised Schur module $\scS_D(\bbC^n)$.
\end{Theorem}
\begin{proof}
	Since the recurrences \cref{lemma:character-of-partition-sequence-diagram} and \cref{lemma:character-of-partition-sequence-multiset} are identical, letting $J = J(\underline{\lambda})$ we have the equality of characters $\ch \scS_D^\flag(V_\bullet) = \ch \cM(\GL_r, \bR, J)$. Applying the Demazure operator $\pi_{w_\circ}$ to both sides yields (by \cref{corollary:demazure-stabilisation} and \cite[Theorem~21]{reinerPercentageAvoidingNorthwestShapes1998}) the equality of characters $\ch \scS_D(\bbC^r) = \ch \cM(\GL_r, \bR)$. Since both $D$ and $\bR$ are stable for $\GL_r$ we have $c_\bR^\mu = c_D^\mu$ for all partitions $\mu$. Comparing the two restriction rules given in \cref{corollary:monomial-restriction-rule} and \cref{section:definition-generalised-schur} completes the proof.
\end{proof}

\begin{Remark}
	\cref{theorem:crystal-of-schur-module} applies to all finite multisets $\bR$ since we may assume that $\bR$ is contained in $\down(\{(1, -1)\})$ after performing a vertical shift on the whole of $\bR$, which does not change the isomorphism class of $\cM(\bR)$. We may always write the shifted $\bR$ as $\bR(\underline{\lambda})$ for some partition sequence $\underline{\lambda}$, then $\cM(\GL_n, \bR)$ is the crystal of $\scS_{D(\underline{\lambda})}(\bbC^n)$. Conversely, every column-convex diagram is $D(\underline{\lambda})$ for some $\lambda$, and hence \cref{theorem:crystal-of-schur-module} gives a positive combinatorial formula for $\ch \scS_D(\bbC^n)$ in terms of a sum over elements of the corresponding product monomial crystal.
\end{Remark}
\begin{Remark}
	By \cite{kamnitzerCategoryAffineGrassmannian2019}, the category $\cO$ of truncated shifted Yangians provides categorifications of $\fg$-modules whose associated crystal is the product monomial crystal. Hence, by \cref{theorem:crystal-of-schur-module}, we deduce that in type $A$ these are categorifications of generalised Schur modules, in the column-convex case. In particular, the truncated shifted Yangians produce (the first known) categorifications of skew Schur modules.
\end{Remark}

\printbibliography

\end{document}